\documentclass[11pt]{amsart}
\usepackage{amsmath,amsthm}
\usepackage{graphics}
\usepackage{latexsym}
\usepackage{verbatim}
\usepackage{amsmath}
\usepackage{amsthm}
\usepackage{amssymb}
\usepackage{epsfig}
\usepackage{epstopdf}
\usepackage{hyperref}
\usepackage[]{hyperref}
\hypersetup{urlcolor=blue, citecolor=red}
\usepackage{dsfont}
\usepackage[table]{xcolor}
\usepackage{multirow}
\usepackage{float}
\usepackage{tikz}
\usetikzlibrary{arrows, calc, shapes, positioning, decorations.pathreplacing}
\usepackage{comment}
\usepackage{subfigure}
\usepackage{MnSymbol}
\numberwithin{equation}{section}

\newcommand{\Trperp}{\textrm{\normalfont Tr}_\perp}
\newcommand{\mypar}{{\mkern3mu\vphantom{\perp}\vrule depth 0pt\mkern2mu\vrule depth 0pt\mkern3mu}}
\newcommand\ds{\displaystyle}
\newcommand\dD{\textrm{\normalfont d}}

\ifpdf
  \DeclareGraphicsExtensions{.eps,.pdf,.png,.jpg}
\else
  \DeclareGraphicsExtensions{.eps}
\fi
\newcommand{\mE}{\mathcal{E}}

\newcommand{\mR}{\mathcal{R}}

\newcommand\bA{{\mathbf A}}
\newcommand\bB{{\mathbf B}}

\newcommand\bE{{\mathbf E}}
\newcommand\bF{{\mathbf F}}

\newcommand\bH{{\mathbf H}}

\newcommand\bU{{\mathbf U}}

\newtheorem{theorem}{Theorem}[section]
\newtheorem{proposition}[theorem]{Proposition}

\newtheorem{remark}[theorem]{Remark}

\newcommand\be{ \mathbf{e}}
\newcommand\bx{ \mathbf{x} }

\newcommand\bv{ \mathbf{v} }
\newcommand\bw{ \mathbf{w} }
\newcommand\vp{ \varphi }
\newcommand\kk{ \mathbf{k} }
\newcommand{\RR}{{\mathbb R}}
\newcommand{\eps}{\varepsilon}

\newcommand{\ZZ}{{\mathbb Z}}

\title[Numerical simulations to the Vlasov-Poisson system with
a strong magnetic field]{Numerical simulations to the Vlasov-Poisson system with
a strong magnetic field}

\author[Francis Filbet  and Chang Yang]{}


\email{francis.filbet@math.univ-toulouse.fr}
\email{yangchang@hit.edu.cn}

\begin{document}
\maketitle

% Enter the first author's name and address:
\centerline{\scshape Francis Filbet}
\medskip
{\footnotesize
    % please put the address of the second  and third author
    \centerline{Institut de Math\'ematiques de Toulouse, UMR5219,
      Universit\'e de Toulouse \& IUF , F-31062}
    \centerline{Toulouse, France}
}

\medskip

\centerline{\scshape Chang Yang}
\medskip
{\footnotesize
    % please put the address of the second  and third author
    \centerline{Department of Mathematics,  Harbin Institute of Technology,}
    \centerline{ 92 West Dazhi Street, Nan Gang District, Harbin 150001, China}
}
%  The abstract of your paper
%%%%%%%%%%%%%%%%%%%%%%%%%%%%%%%%%%%%%%%%%%%%%%%%%%%%%%%%%%%%%%%%%%%%%%%%%%%%%%%%%%%%%%%%%%%%
%%%%%%%%%%%%%%%%%%%%%%%%%%%%%%%%%%%%%%%%%%%%%%%%%%%%%%%%%%%%%%%%%%%%%%%%%%%%%%%%%%%%%%%%%%%%

\begin{abstract}
In this paper, we present a Particle-In-Cell algorithm based on
semi-implicit/explicit time discretization techniques for the simulation of
the three dimensional Vlasov-Poisson system in the presence of a
strong external magnetic field. When the intensity of the magnetic
field is sufficiently large and for any time step, the numerical scheme provides formally a
consistent approximation of  the drift-kinetic
model, which corresponds to the asymptotic model.  Numerical results show that this new Particle-In-Cell 
method is efficient and accurate for large time steps.
\end{abstract}

%%%%%%%%%%%%%%%%%%%%%%%%%%%%%%%%%%%%%%%%%%%%%%%%%%%%%%%%%%%%%%%%%%%%%%%%%%%%%%%
%%%%%%%%%%%%%%%%%%%%%%%%%%%%%%%%%%%%%%%%%%%%%%%%%%%%%%%%%%%%%%%%%%%%%%%%%%%%%%%

\tableofcontents

\section{Introduction}
\label{sec:1}
\setcounter{equation}{0}

Magnetized plasmas are encountered in a wide variety of astrophysical
 situations, but also  in magnetic fusion devices such as tokamaks,
 where a large external magnetic field needs to be applied in order to keep
  the particles on the desired tracks. Such a dynamic can be  described by the Vlasov-Poisson equation, where  plasma particles evolve under self-consistent electrostatic field and the confining magnetic
field.

We assume that on the time scale we consider, collisions can be neglected both for ions and electrons, hence collective effects are dominant and the plasma is entirely modelled with kinetic transport equations, where the unknown is the number density of particles $f (t, \bx, \bv)$ depending on time $t \geq 0$, position $\bx \in \RR^3$ and velocity $\bv \in \RR^3$. Such a kinetic model provides an appropriate description of turbulent transport in a fairly general context, but it requires to solve a six dimensional problem which leads to a huge computational cost.

On the one hand, many asymptotic models with a smaller number of variables than the kinetic description were developed. For instance, large magnetic fields usually lead
to the so-called drift-kinetic limit~\cite{Hazeline1978, Hazeline2003,
bib:degond-2}
and for a mathematical point of view \cite{FS:00, SR:02, Bostan:18, Bostan:18-1}. In this regime, due to the large applied magnetic field,
particles are confined along the magnetic field lines and their period of rotation around these lines (called the cyclotron period) becomes small. However, such a reduced model only valid with strong magnetic field assumption, hence it could not describe all physics of magnetized plasma.
On the other hand, in some recent work~\cite{xiao2015, bibQin2016},
numerical methods are developed for the full kinetic models, such as
the Vlasov-Maxwell equation. For example, in \cite{bibQin2016} and
\cite{Sonnen0, Sonnen1}, the authors have developed a symplectic Particle-In-Cell
method, which can preserve the geometrical structure of the system,
hence this property may help  to preserve the accuracy for long time simulation. However, in our
context, this scheme is not necessarily efficient due to its
complexity and the limitation on the time step since it requires the
time resolution of all time scales.

Another approach with similar advantages, developed in \cite{bibCLM,
  bibCLMZ, bibCLMZ2} and \cite{CFHM:15, FHLS:15}, consists in
explicitly doubling time variables and seeking higher-dimensional
partial differential equations and boundary conditions in variables
$(t,\tau,\bx,\bv)$ that contains the original system at the
$\eps$-diagonal $(t,\tau)=(t,t/\eps)$ where $\eps$ represents for
instance the ratio between the plasma and the cyclotron frequencies. While the corresponding methods are extremely good at capturing oscillations their design require a deep \emph{a priori} understanding of the detailed structure of oscillations. 

In the very recent works of Filbet and Rodrigues~\cite{bibFilbet_Rodrigues,bibFilbet_Rodrigues2}, a new asymptotically stable Particle-In-Cell methods are developed in the request of efficiency for full kinetic models. The numerical methods are developed in the two-dimensional framework, where one restricts to the perpendicular dynamics. On the one hand, this numerical methods contain the efficient property of the Particle-In-Cell (PIC) method; on the other hand, this numerical method is free from the stiffness of the full kinetic system. Moreover, up to third order method is also proposed, thus this method is very accurate for long term simulations.

In this paper, we  extend the asymptotically stable Particle-In-Cell methods for three dimensional Vlasov-Poisson equation. For clarity, we first consider a cylindrical geometry with uniform external magnetic filed. Under this assumption,  by following the main lines in~\cite{bibFilbet_Rodrigues}, we develop implicit numerical methods for the characteristic curve system. More precisely, though the numerical methods are implicit, only the stiff terms (characterized by $1/\eps$) are implicitly computed, and the other terms can be explicitly computed. Then by reformulating the numerical methods and dropping the second order terms with respect to $\eps$, we derive the numerical methods for the characteristic curve system corresponding to the Drift-Kinetic model. We can formally show the solutions of these two systems are second order consistent with respect to $\eps$.

The rest of the paper is organized as follows. In Section~\ref{sec:2}, we derive the Vlasov-Poisson equation in our interested scaling, and develop its second order consistent non-stiff model, the drift-Kinetic model. In Section \ref{sec:3}
we present several time discretization techniques based on high-order
semi-implicit schemes \cite{BFR:15} for  the Vlasov-Poisson system with a
strong external magnetic field, and we prove consistency of
the schemes even when the intensity of the magnetic field becomes
large with preservation of the
order of accuracy (from first to third order accuracy).  
Section \ref{sec:5} is then devoted to numerical simulations for one
single particle motion and for the Vlasov-Poisson model for various
asymptotics, which illustrate the
advantage of high order schemes. Finally in Section~\ref{sec:6}, we
conclude the paper and give the perspectives.

%%------------------------------------------------
\section{Mathematical models}
\label{sec:2}
\setcounter{equation}{0}

In this section, we will introduce the models to describe the electrostatic perturbations of spatially non-uniform plasmas.
The Vlasov equation for the ion distribution function $f$ in standard form in
standard notation  is
\begin{equation}
  \frac{\partial f}{\partial
    t}\,+\,\bv\cdot\nabla_\bx f\,+\,\frac{q}{m}\left(\bE+\bv\wedge\bB_{\rm
      ext}\right)\cdot\nabla_\bv f\,=\,0,
  \label{eq:vlasov}
\end{equation}
where $t\in\RR^+$ is the time variable,
$\bx\in\Omega\subset\RR^3$ is the space variable, 
$\bv\in\RR^3$ is the  velocity variable, $m$ is the ion particle mass, $q$ is its charge, $\bE=-\nabla_{\bx} \phi(t,\bx)$ is the electric field and
$\bB_{\rm ext}$ is the external magnetic field. The potential $\phi$
is solution to the Poisson equation
\begin{equation}
\label{eq:poisson}
- \epsilon_0\,\Delta_\bx\phi \,=\, {\rho},
\end{equation}
where $\epsilon_0$ represents the permittivity of vacuum and $\rho$
is the charge density 
$$
\rho(t,\bx) \,:=\, q \,\int_{\RR^3} f(t,\bx,\bv)\,\dD \bv \,-\, \rho_0,
$$
with $\rho_0$ the charge density of a fixed species.
For practical applications,  this model has to be supplemented with
suitable boundary conditions. Here we will consider a cylindrical domain of the form
\begin{equation*}
 \Omega=\left\{(x,y,z)\in\mathbb{R}^3; \,\,(x,y)\in D, \,0\,\leq \,z\,\leq\, L_z\right\},
\end{equation*}
where $D$ an arbitrary two dimensional domain (disk and D-shaped domain
will be used). We assume that the electric potential is periodic in
the $z$ variable and vanishes at the boundary $\partial D$
\begin{equation}
 \phi(\mathbf{x})\,=\,0,\quad\mathbf{x}\in\partial D\times[0,L_z].
 \label{eq:DK_BC1}
\end{equation}
Furthermore, we assume that the plasma is well confined hence the
distribution function also vanishes on $\partial D$ and is periodic in
$z$.

\subsection{Characteristic curves}
Here, for simplicity we set all physical constants to one and consider
that $\eps > 0$ is a small parameter related to the ratio between the
reciprocal Larmor frequency and the advection time scale. We refer to
\cite{bib:degond-2} for more details on the scaling issues on this problem.

 Let us now consider the magnetic field    has a fixed direction
$\bB_{\rm ext}\,=\,\eps^{-1}\,b(t,\bx_\perp)\,\be_z$, where the vector $ \be_z$ stands for the unit vector in the
toroidal direction. Then the characteristic curves corresponding to
the Vlasov equation \eqref{eq:vlasov} are given by
\begin{equation*}
\left\{
\begin{array}{lll}
\displaystyle\frac{\dD \bx}{\dD t} &=& \bv,\\  [3.5mm]
\displaystyle\frac{\dD\bv}{\dD t} &=& \bE(t, \bx) + \bv\wedge\bB_{\rm
                                      ext}(t,\bx). 
\end{array}
\right.
\end{equation*}
The goal is to identify the fast and slow variable, then to isolate
the stiffest scale to keep only  the slow scale. Therefore, we
introduce a decomposition according to the parallel direction to $\be_z$ and its orthogonal direction
$$
\left\{
\begin{array}{l}
v_\mypar \,=\,\langle\bv\,,\,\be_z\rangle\,=\, v_z,
\\
\,
\\
\bv_\perp=\bv-v_\mypar\,\be_z\,=\, ^t(v_x,v_y,0),
\end{array}\right.
$$
where $\langle.,.\rangle$ denotes the scalar product in $\RR^3$ then
we proceed similarly for the electric field $\bE=\bE_\perp + E_\mypar \be_z$ and the space component
$\bx=\bx_\perp + x_\mypar\,\be_z \in\RR^3$. Thus the system of the characteristic curves now becomes
\begin{equation}\label{eq:characteristic_scaled_vlasov2}
\left\{
\begin{array}{lll}
\displaystyle\frac{\dD \bx_\perp}{\dD t} &=& \bv_\perp,\\  [3.5mm]
\displaystyle\frac{\dD x_{\mypar}}{\dD t} &=& v_{\mypar},\\  [3.5mm]
\displaystyle\frac{\dD\bv_\perp}{\dD t} &=& \bE_\perp (t, \bx) -\displaystyle b(t,\bx_\perp)\,\frac{\bv_\perp^\perp}{\varepsilon}, \\ [3.5mm]
\displaystyle\frac{\dD v_{\mypar}}{\dD t} &=& E_{\mypar} (t, \bx),
\end{array}
\right.
\end{equation}
where we used the notation $\bv^\perp_\perp =\, ^t (-v_y, v_x,0)$. 
\subsection{Asymptotic analysis} 
In the sequel we replace the system
\eqref{eq:characteristic_scaled_vlasov2} by an equivalent system where
we separate the fast variable $\bv_\perp$ and the slow ones.
 Therefore we first set
\begin{equation}
\label{def:F}
\bF(t,\bx) \,:=\, \frac{\bE(t,\bx)}{b(t,\bx_\perp)}
\end{equation} 
and using the third equation of
\eqref{eq:characteristic_scaled_vlasov2}, we may write it in a
different manner as
\begin{equation}
\label{eq:vperp}
\left\{
\begin{array}{l}
\displaystyle\frac{\dD\bv_\perp}{\dD t} \,=\, \bE_\perp (t, \bx) - b(t,\bx_\perp)\,\frac{\bv_\perp^\perp}{\eps},
\\ \, \\
\ds\frac{\dD}{\dD t}\left(
  \frac{\bv_\perp^\perp}{b(t,\bx_\perp)}\right) \,=\,
  \left(\bF_\perp(t, \bx) - \frac{\partial_tb \,+\,
  \langle\nabla_{\bx_\perp}b,\bv_\perp\rangle}{b^2(t,\bx_\perp)}\,\bv_\perp\right)^\perp\,+\, \frac{\bv_\perp}{\eps}.
\end{array}\right.
\end{equation}
This last formulation will help us to  separate the different
scales.

On the one hand, we combine the first equation in
\eqref{eq:characteristic_scaled_vlasov2} and the second equation in
(\ref{eq:vperp}), which gives
\begin{equation}
\label{one}
\frac{\dD}{\dD t}\left(\bx_\perp -  \eps\,\frac{\bv^\perp_\perp}{b(t,\bx_\perp)} \right) \,=\, -\eps\,\left(\bF_\perp(t, \bx) - \frac{\partial_tb \,+\,
  \langle\nabla_{\bx_\perp}b,\bv_\perp\rangle}{b^2(t,\bx_\perp)}\,\bv_\perp\right)^\perp.
\end{equation}
On the other hand we define $e_\perp$ as the local kinetic energy given by
\begin{equation}
\label{def:eperp}
e_\perp = \frac{\|\bv_\perp\|^2}{2},
\end{equation} 
hence using the orthogonality between $\bv_\perp^\perp$ and
$\bv_\perp$, the kinetic energy variable $e_\perp$ is solution to 
$$
\frac{\dD e_\perp}{\dD t} \,=\, \langle\bE_\perp(t,\bx),
                              \bv_\perp\rangle.
$$
Then we use the second equation in (\ref{eq:vperp})   and substitute it into the equation for $e_\perp$, it yields
$$
\frac{\dD e_\perp}{\dD t} \,=\, \eps\, \left(\frac{\partial_t b  +
  \langle\nabla_{\bx_\perp}b,\bv_\perp\rangle}{b^2(t,\bx_\perp)}\right) \,\langle\bE_\perp(t,\bx),
\bv_\perp^\perp\rangle \,+\, \eps\,\left\langle \bE_\perp(t,\bx),
\frac{\dD}{\dD t}\left(\frac{\bv_\perp^\perp}{b(t,\bx_\perp)} \right)\right\rangle, 
$$
which may be written as
\begin{eqnarray}
\nonumber
\frac{\dD }{\dD t}\left(e_\perp - \eps\,{\left\langle\bF_\perp(t,\bx),\bv_\perp^\perp\right\rangle} \right) &=&
\eps\, \left(\frac{\partial_t b  +
  \langle\nabla_{\bx_\perp}b,\bv_\perp\rangle}{b^2(t,\bx_\perp)}\right) \,\langle\bE_\perp(t,\bx),
\bv_\perp^\perp\rangle
\\ 
\nonumber
& &  -\eps\,\left\langle \partial_t\bE_\perp+\dD_\bx \bE_\perp\,\bv,\frac{\bv_\perp^\perp}{b(t,\bx_\perp)}\right\rangle
\\
&=& -\eps\,\langle \partial_t\bF_\perp+\dD_\bx
    \bF_\perp\,\bv ,\bv_\perp^\perp\rangle. 
\label{two}
\end{eqnarray}
Thus gathering \eqref{one} and \eqref{two}, we get the following
system of equations
\begin{equation}
\label{three}
\left\{
\begin{array}{l}
\ds \frac{\dD}{\dD t}\left(\bx_\perp -  \eps\,\frac{\bv^\perp_\perp}{b(t,\bx_\perp)} \right) \,=\, -
  \eps\,\left( \bF_\perp(t,\bx) \,-\,  \frac{\partial_t b+\langle
  \nabla_{\bx_\perp}b,\bv_\perp \rangle}{b^2(t,\bx_\perp)}\,\bv_\perp\right)^\perp,
\\ \, \\ 
\ds\frac{\dD }{\dD t}\left(e_\perp - \eps\,{\left\langle\bF_\perp(t,\bx),\bv_\perp^\perp\right\rangle} \right) \,=\,
-\eps\,\langle \partial_t\bF_\perp+\dD_\bx \bF_\perp(t,\bx) \,\bv,\bv_\perp^\perp\rangle.
\end{array}\right.
\end{equation}
This last system on the new variables $\bx_\perp -  \eps\,\bv^\perp_\perp/b(t,\bx_\perp)$ and $e_\perp -\eps\,\langle\bF_\perp(t,\bx),\bv_\perp^\perp\rangle$ is
  an improvement compared to the equations on $(\bx_\perp,e_\perp)$ since it only
  contains terms of order $\eps$ in the right hand side, hence it means
  that it evolves slowly. However the price to pay is that it now involves
  bilinear terms with respect to the fast variable $\bv_\perp$, which
  have to be controlled.
 \begin{proposition}
\label{prop:0}
For any $T>0$, consider  $\bA \in W^{1,\infty}(0,T)$ with $\bA(t)\in
\mathcal{M}_{3,3}(\RR)$ for any $t\in [0,T]$ and $\bv_\perp$ the solution to (\ref{eq:vperp}).  Then 
\begin{eqnarray}
\label{res:1}
\left\langle\bv_\perp, \bA(t)\bv_\perp\right\rangle
  &=&e_\perp\,\Trperp(\bA(t)) \,+\,\frac{\eps}{2}\frac{\dD}{\dD t}\left[\frac{1}{b(t,\bx_\perp)}\left\langle \bv_\perp,
    \bA(t)\,\bv_\perp^\perp\right\rangle\right]
\\
 &-&\frac{\eps}{2}\,\left[\left\langle \bF_\perp,
                                                     \bA(t)
                                                     \bv_\perp^\perp\right\rangle
                                                     + \left\langle \bv_\perp,
                                                     \bA(t) \bF_\perp^\perp
                                                     \right\rangle
                                                     \right]
\nonumber
\\
&-& \frac{\eps}{2\,b(t,\bx_\perp)}\,\left[  \langle \bv_\perp,
                                                     \left(\bA^\prime(t)
    - \partial_t\log(b)\,\bA(t)\right)
                                                     \bv_\perp^\perp\rangle\right]
\nonumber
\\    
&+&\frac{\eps}{2\,b^2(t,\bx_\perp)} \, \langle\nabla_{\bx_\perp}b(t,\bx_\perp),\bv_\perp\rangle \, \langle \bv_\perp,
                                                     \bA(t)
                                                     \bv_\perp^\perp\rangle,
\nonumber
\end{eqnarray}
where where $\Trperp$ denotes the trace of the part on the plane
orthogonal to $\be_z$, that is,
$$
\Trperp \bA(t) \,=\, \langle \be_x ,\bA(t) \, \be_x\rangle \,+\,  \langle \be_y, \bA(t) \, \be_y\rangle
$$
\end{proposition}
\begin{proof} 
For any $t\in [0,T]$  we choose $\bA(t)\in \mathcal{M}_{3,3}(\RR)$ a square
matrix.  On the one hand using the two equations in \eqref{eq:vperp}, we get that
\begin{eqnarray*}
\frac{\dD}{\dD t}\left[\frac{\eps}{b(t,\bx_\perp)}\left\langle \bv_\perp,
    \bA(t)\,\bv_\perp^\perp \right\rangle\right] &=&                                 
                                                     \left\langle \eps\,\bF_\perp-\bv_\perp^\perp,
                                                     \bA(t)
                                                     \bv_\perp^\perp\right\rangle
\,+\,\frac{\eps}{b(t,\bx_\perp)}
                                                     \langle \bv_\perp,
                                                     \bA^\prime(t)
                                                     \bv_\perp^\perp\rangle
                                                     \\
&+& \,   \left\langle \bv_\perp,\bA(t) \left[ \eps\left(\bF_\perp(t, \bx) - \frac{\partial_tb \,+\,
  \langle\nabla_{\bx_\perp}b,\bv_\perp\rangle}{b^2(t,\bx_\perp)}\,\bv_\perp\right)^\perp\,+\,
    \bv_\perp\right]\right\rangle.
\end{eqnarray*}
Then after reordering, it yields 
\begin{eqnarray*}
\frac{\dD}{\dD t}\left[\frac{\eps}{b(t,\bx_\perp)}\left\langle \bv_\perp,
    \bA(t)\,\bv_\perp^\perp \right\rangle\right] &=&   \eps\,\left[\left\langle \bF_\perp,
                                                     \bA(t)
                                                     \bv_\perp^\perp\right\rangle
                                                     + \left\langle \bv_\perp,
                                                     \bA(t) \bF_\perp^\perp
                                                     \right\rangle
                                                     \right]
\\
&+& \frac{\eps}{b(t,\bx_\perp)}\,\left[  \langle \bv_\perp,
                                                     \left(\bA^\prime(t)
    - \partial_t\log(b)\,\bA(t)\right)
                                                     \bv_\perp^\perp\rangle\right]
\\    
&-&\frac{\eps}{b^2(t,\bx_\perp)} \, \langle\nabla_{\bx_\perp}b(t,\bx_\perp),\bv_\perp\rangle \, \langle \bv_\perp,
                                                     \bA(t)
                                                     \bv_\perp^\perp\rangle
\\
&-&            \langle \bv_\perp^\perp,
                                                     \bA(t)
                                                     \bv_\perp^\perp\rangle
    + \langle \bv_\perp,
                                                     \bA(t)
                                                     \bv_\perp\rangle.
\end{eqnarray*}
On the other hand, we observe by definition of $\Trperp$
$$
\left\langle\bv_\perp, \bA(t)\bv_\perp\right\rangle  + \left\langle \bv_\perp^\perp,
  \bA(t) \bv_\perp^\perp\right\rangle \,=\,
{\|\bv_\perp\|^2}\,\Trperp(\bA(t)).
$$
Recalling that $e_\perp=\|\bv_\perp\|^2/2$, a suitable reduction is thus
\begin{eqnarray*}
\left\langle\bv_\perp, \bA(t)\bv_\perp\right\rangle
  &=&e_\perp\,\Trperp(\bA(t)) \,+\,\frac{\eps}{2}\frac{\dD}{\dD t}\left[\frac{1}{b(t,\bx_\perp)}\left\langle \bv_\perp,
    \bA(t)\,\bv_\perp^\perp\right\rangle\right]
\\
 &-&\frac{\eps}{2}\,\left[\left\langle \bF_\perp,
                                                     \bA(t)
                                                     \bv_\perp^\perp\right\rangle
                                                     + \left\langle \bv_\perp,
                                                     \bA(t) \bF_\perp^\perp
                                                     \right\rangle
                                                     \right]
\\
&-& \frac{\eps}{2\,b(t,\bx_\perp)}\,\left[  \langle \bv_\perp,
                                                     \left(\bA^\prime(t)
    - \partial_t\log(b)\,\bA(t)\right)
                                                     \bv_\perp^\perp\rangle\right]
\\    
&+&\frac{\eps}{2\,b^2(t,\bx_\perp)} \, \langle\nabla_{\bx_\perp}b(t,\bx_\perp),\bv_\perp\rangle \, \langle \bv_\perp,
                                                     \bA(t)
                                                     \bv_\perp^\perp\rangle.
\end{eqnarray*}
\end{proof}

Let us now apply Proposition \ref{prop:0} to the system
\eqref{three}. For the first equation we choose $\bA$ such that 
$$
\langle \bv_\perp, \bA(t)\bv_\perp\rangle \,=\,
\frac{\eps}{b^2(t,\bx_\perp)} \, \langle
\nabla_{\bx_\perp}b,\bv_\perp\rangle \,\langle \bv_\perp,
\be_{\alpha}\rangle, 
$$
with $\alpha\in\{x,\,y\}$. Then there exist two bounded functions
$\Theta_1$ and $\Sigma_1$ such that
\begin{eqnarray}
\label{bene}
\frac{\dD}{\dD t}\left(\bx_\perp -
  \eps\,\frac{\bv^\perp_\perp}{b(t,\bx_\perp)} + \eps^2
  \Theta_1(t,\bx,\bv)\right) &=&  -\eps\,\left( \bF_\perp(t,\bx) 
  \,-\,  \frac{e_\perp}{b^2(t,\bx_\perp)}\,
  \nabla_{\bx_\perp}b\right)^\perp \\
&+&\, \eps\frac{\partial_t b}{b^2} \,\bv_\perp^\perp +
\eps^2 \,\Sigma_1(t,\bx,\bv).
\nonumber
\end{eqnarray}
From the system \eqref{three} and using that $\bv_\perp$ is uniformly
bounded with respect to $\eps$, we get that $\bv_\perp$ weakly
converges to zero, for $\eps\ll 1$. Therefore,  removing high
order terms (larger than one) in \eqref{bene}, we get an approximated
equation given by
\begin{equation}
\label{eq:lim1}
\displaystyle\frac{\dD \bx_\perp }{\dD t} \,=\, -\eps\,\left( \bF_\perp(t,\bx) 
  \,-\,  \frac{e_\perp}{b^2(t,\bx_\perp)}\,
  \nabla_{\bx_\perp}b(t,\bx_\perp)\right)^\perp. 
\end{equation}
This equation corresponds to the guiding center approximation.

Next we treat the second equation of \eqref{three} by applying
Proposition \ref{prop:0} with $\bA=\dD_{\bx}\bF_\perp$, then there
exist  two bounded functions
$\Theta_2$ and $\Sigma_2$ such that
\begin{eqnarray*}
\frac{\dD }{\dD t}\left(e_\perp -
  \eps\,{\left\langle\bF_\perp(t,\bx),\bv_\perp^\perp\right\rangle} +
  \eps^2 \Theta_2(t,\bx,\bv)\right) &=&
\eps\, {\rm div}_{\bx_\perp} \bF^\perp_\perp(t,\bx) \,e_\perp  
\\
&-&\eps\,\langle \partial_t\bF_\perp-v_z\partial_z\bF_\perp(t,\bx),\bv_\perp^\perp\rangle + \eps^2\Sigma_2(t,\bx,\bv).
\end{eqnarray*}
Neglecting high order terms, it yields
\begin{equation}
\label{eq:lim2}
\frac{\dD e_\perp }{\dD t} \,=\,
+\eps\, {\rm div}_{\bx_\perp} \bF^\perp_\perp(t,\bx) \,e_\perp.  
\end{equation}
Gathering the results and neglecting the terms of order larger than one, we get an
approximation to system of the characteristic curves
\eqref{eq:characteristic_scaled_vlasov2},
\begin{equation}\label{eq:characteristic_lim}
\left\{
\begin{array}{lll}
\displaystyle\frac{\dD \bx_\perp}{\dD t} &=& -\ds\eps\,\left( \bF_\perp(t,\bx) 
  \,-\,  \frac{e_\perp}{b^2(t,\bx_\perp)}\,
  \nabla_{\bx_\perp}b(t,\bx_\perp)\right)^\perp,\\  [3.5mm]
\displaystyle\frac{\dD x_{\mypar}}{\dD t} &=& v_{\mypar},\\  [3.5mm]
\displaystyle\frac{\dD e_\perp}{\dD t} &=& \eps\, {\rm div}_{\bx_\perp} \bF_\perp^\perp(t,\bx) \,e_\perp,\\ [3.5mm]
\displaystyle\frac{\dD v_{\mypar}}{\dD t} &=& E_{\mypar} (t, \bx).
\end{array}
\right.
\end{equation}
This systems only contains the information on slow scales and
corresponds to the characteristic curves of the drift-kinetic equation
\begin{equation}
  \frac{\partial F}{\partial
    t}\,+\, \bU^{\rm gc}\cdot\nabla_\bx F \,+\, \eps\,{\rm
    div}_{\bx_\perp} \bF_\perp^\perp(t,\bx)\,e_\perp \,\frac{\partial F}{\partial {e_\perp}}
  \,+\, E_\mypar\frac{\partial F}{\partial{v_\mypar} } \,=\,0,
  \label{eq:drift}
\end{equation}
where the guiding center velocity $\bU^{\rm gc}$ is given by
\begin{equation}
\label{def:Ugc}
\left\{
\begin{array}{l}
\bU^{\rm gc}(t,\bx,e_\perp,v_\mypar) \,:=\, v_\mypar \be_z + \bU^{\rm
  gc}_\perp(t,\bx,e_\perp),
\\ \, \\
\ds \bU^{\rm
  gc}_\perp(t,\bx,e_\perp) = \,-\,  \eps\,\left( \bF_\perp(t,\bx) 
  \,-\,  \frac{e_\perp}{b^2(t,\bx_\perp)}\,
  \nabla_{\bx_\perp}b(t,\bx_\perp)\right)^\perp.
\end{array}\right.
\end{equation}
\subsection{Reformulation of the model \eqref{eq:characteristic_scaled_vlasov2}} 
The aim of the paper is to construct a particle method which is stable
and consistent for  $\eps\ll 1$, that is, when the
magnetic field becomes large. Therefore, we need to
ensure that the approximation of the velocity variable $\bv_\perp$ tends to zero
when $\eps\rightarrow 0$ and also the modulus $\|\bv_\perp\|^2/2$ converges towards
an approximation of $e_\perp$ solution to the third equation in
\eqref{eq:characteristic_lim}. Hence, we  reformulate the initial problem
\eqref{eq:characteristic_scaled_vlasov2}   and introduce an additional variable $e_\perp$. More
precisely, we replace \eqref{eq:characteristic_scaled_vlasov2} by an augmented
system as
\begin{equation}
\label{eq:00}
\left\{
\begin{array}{lll}
\ds\frac{\dD \bx}{\dD t} &=& \bv,\\  [3.5mm]
\displaystyle\frac{\dD e_\perp}{\dD t} &=& \langle\bE_\perp(t,\bx),\bv_\perp\rangle,\\  [3.5mm]
\displaystyle\frac{\dD\bv}{\dD t} &=& \bH(t, \bx,\bv_\perp,e_\perp)
                                      -\displaystyle
                                      b(t,\bx_\perp)\,\frac{\bv_\perp^\perp}{\varepsilon}, 
\end{array}
\right.
\end{equation}
where the force field is chosen as
\begin{equation}
\label{def:H}
\bH(t, \bx,\bv_\perp,e_\perp)  :=
\bE(t,\bx)\,-\,\chi(e_\perp,\bv_\perp)\,\nabla_{\bx_\perp} \ln b(t,\bx_\perp)
\end{equation}
and the function $\chi\in W^{1,\infty}_{loc}(\RR^4)$ is  such that  for any $e_\perp\in\RR$ and
$\bv_\perp=(v_x,v_y,0)\in\RR^3$, 
\begin{equation}
\label{hyp:2}
\left\{
\begin{array}{l}
\ds\chi\left(\|\bv_\perp\|^2/2\,,\, \bv_\perp\right) =0,\\ \,\\
\ds \chi(e_\perp,0_{\RR^3}) = e_\perp,  \\ \, \\
\ds 0 \leq \chi(e_\perp,\bv_\perp) \leq e_\perp.
\end{array}\right.
\end{equation}
For concreteness, in the following, we actually choose $\chi$ as
\begin{equation}
\chi(e_\perp,\bv_\perp) = \frac{e_\perp}{e_\perp + \|\bv_\perp\|^2/2} \left( e_\perp - \frac{\|\bv_\perp\|^2}{2} \right)^+,
\quad \forall (e_\perp,\bv_\perp)\in\RR\times\RR^3,
\end{equation}
where $s^+ = \max(0,s)$.

Clearly, when we choose initially $e_\perp=\|\bv_\perp\|^2/2$,  the
second equation on $e_\perp$ of \eqref{eq:00} can be deduced from the
third by multiplying it by $\bv_\perp$. Hence, the
solution of the augmented system \eqref{eq:00}  is also solution to
the initial system of the characteristic curves
\eqref{eq:characteristic_scaled_vlasov2}. At the discrete level, this
last formulation \eqref{eq:00} will be more suitable to construct a numerical
approximation which is consistent in the limit $\eps \rightarrow 0$,
that is, the approximation $\bv_\perp\rightarrow 0$ whereas $(\bx, e_\perp,v_\mypar)$ is consistent with
the solution of the asymptotic model   \eqref{eq:characteristic_lim},
when $\eps\ll 1$.

In the sequel, we assume that the electric field $\bE$ is such that 
\begin{equation}
\bE=-\nabla_\bx\phi \in  W^{1,\infty}((0,T)\times\RR^3)
\label{hyp:E}
\end{equation} 
and the intensity $b$ of the magnetic field in the direction $\be_z$ is such that 
\begin{equation}
\label{hyp:b}
b(t,\bx_\perp)>b_0>0,\quad {\rm and}\quad b \in
W^{2,\infty}((0,T)\times\RR^3).
\end{equation}

We define the operator $\mR^n$ such that $\bv_\perp^\perp := \mR^n \bv_\perp$,
that is,
$$
\mR^n \,:=\, b(t^n,\bx_\perp^n)\,
\left( \begin{array}{lll}
0 & -1 & 0 \\
1 & 0 & 0 \\
0 & 0 & 0
\end{array}\right).
$$ 
%%%%%%%%%%%%%%%%%%%%%%%%%%%%%%%%%%%%%%%%%%%%%%%%%%%%%%%%%%%%%%%%%%%%%%%%%%%%%%%%%%%%%
%%%                                                                                                                                                                                  %%%
%%%%%%%%%%%%%%%%%%%%%%%%%%%%%%%%%%%%%%%%%%%%%%%%%%%%%%%%%%%%%%%%%%%%%%%%%%%%%%%%%%%%% 

\section{A particle method for Vlasov-Poisson
  system with
  a strong magnetic
  field}
\label{sec:3}
\setcounter{equation}{0}

The numerical resolution of the Vlasov equation and related models is usually performed by
Particle-In-Cell (PIC) methods which approximate the plasma by a
finite number of particles. Trajectories of these particles are
computed from characteristic curves (\ref{eq:characteristic_scaled_vlasov2}) corresponding to
the the Vlasov equation (\ref{eq:vlasov}), whereas self-consistent fields are computed on a mesh of the physical
space. We refer the reader to \cite{birdsall, bib:degond-1} for a thorough discussion and other applications to
plasma physics, or to~\cite{bibFilbet_Rodrigues} for a brief review of particle methods.

Let us now develop a
particle method for the Vlasov equation  (\ref{eq:vlasov}), where the key issue is to design a uniformly stable
scheme with respect to the parameter $\eps>0$, which is related to the magnitude
of the external magnetic field. Assume that at time $t^n=n\,\Delta t$,
the set of particles are located in $(\bx_k^n, \bv_k^n)_{1\leq k\leq
  N}$, we want to solve the system \eqref{eq:00} on the time interval $[t^n, t^{n+1}]$,
\begin{equation}
\label{traj:bis}
\left\{
\begin{array}{l}
\ds\frac{\dD\bx_k}{\dD t} = \bv_k, \\  [3.5mm]
\ds\frac{\dD e_{\perp,k}}{\dD t} = \langle\bE_\perp(t,\bx_k),\bv_{\perp,k}\rangle,\\  [3.5mm]

\ds\frac{\dD\bv_k}{\dD t} = {\bH}(t,\bx_k,\bv_{\perp,k},e_{\perp,k}) \,-\,  b(t,\bx_{\perp,k})\frac{\bv_k^\perp}{\eps}, \\  [3.5mm]
\bx_k(t^n) = \bx^n_k, \,\, e_{\perp,k}(t^n) = e^n_{\perp,k}, \,\, \bv_k(t^n) = \bv^n_k,
\end{array}\right.
\end{equation}
where the electric field is computed from a discretization of the Poisson equation
(\ref{eq:poisson}) on a mesh of the physical space.
  
The numerical scheme that we describe here is proposed in the
framework of Particle-In-Cell method, where the solution $f$  is
discretized as follows
$$
f_{N,\alpha}^{n+1}(\bx,\bv) := \sum_{1\leq k \leq N} \omega_k\, \vp_\alpha
(\bx-\bx^{n+1}_k)\, \vp_\alpha(\bv-\bw^{n+1}_k), 
$$
where  $(\bx_k^{n+1},\bw_k^{n+1})$ represents an approximation of
the solution $(\bx_k(t^{n+1}),\bw_k(t^{n+1}))$ to (\ref{traj:bis}),
with
 $$
\bw_k(t) \,:=\, 
                              \sqrt{2e_{\perp,k}(t)} \,
                              \frac{ \bv_{\perp,k}(t) } {\| \bv_{\perp,k}(t) \|} \,+\, v_{\mypar,k}(t)
                              \, \be_z,
$$ 
whereas the function $\vp_\alpha = \alpha^{-3}\vp(\cdot/\alpha)$ is a particle shape function with radius proportional to $\alpha$, usually seen as a smooth approximation of the Dirac measure obtained by scaling a compactly supported "cut-off" function $\vp$ for with common choices include B-splines and smoothing kernels with vanishing moments, see {\it e.g.} \cite{Cottet.Koumoutsakos.2000.cup, Koumoutsakos.1997.jcp}.

When the Vlasov equation (\ref{eq:vlasov}) is coupled with the
Poisson equation (\ref{eq:poisson}), the electric field is computed in a
macro-particle position $\bx^{n+1}_\kk$ at time $t^{n+1}$ as follows
  
\begin{itemize}
\item Compute the density $\rho$
$$
\rho_{h,\alpha}^{n}(\bx) = \sum_{\kk\in\ZZ^3} w_\kk \,\;\vp_\alpha (\bx-\bx^{n}_\kk).
$$
\item Solve a discrete approximation to (\ref{eq:poisson}) 
$$
-\Delta_{h} \phi^n (\bx) = \rho_{h,\alpha}^{n}(\bx).
$$
\item Interpolate the electric field  with the same order of accuracy
  on the points $(\bx_\kk^{n})_{\kk\in\ZZ^3}$.
\end{itemize}

To discretize the system (\ref{traj:bis}),  we apply the strategy developed in
\cite{BFR:15} based on semi-implicit solver for stiff problems. In the
rest of this section, we propose several numerical schemes 
to the system (\ref{traj:bis}) for which the index $k\in\{1,\ldots,N\}$ will
be omitted.   

\subsection{A first order semi-implicit scheme}
For a fixed time step $\Delta t>0$ and a given
electric field ${\bf E}$ and an external magnetic field $\bB_{\rm ext}=\eps^{-1}\,b(t,\bx_\perp)\,\be_z$, we apply a semi-implicit scheme for (\ref{traj:bis}), which is a
combination of backward and forward Euler scheme,
\begin{equation}
\label{scheme:0-2}
\left\{
\begin{array}{l}
\ds\frac{\bx^{n+1} - \bx^n }{\Delta t} \,\,=\, \bv^{n+1},
\\
\,
\\
\ds\frac{e_{\perp}^{n+1} - e_{\perp}^n }{\Delta t} \,=\, \langle
  \bE_\perp(t^n,\bx^n), \bv_{\perp}^{n+1}\rangle,
\\
\,
\\
\ds\frac{\bv^{n+1} - \bv^n }{\Delta t} \,\,=\,
  \bH(t^n,\bx^n,\bv_\perp^n, e_\perp^n) - b(t^n, \bx_\perp^n) \frac{(\bv_\perp^{n+1})^\perp}{\eps},
\end{array}\right.
\end{equation}
where $\bH$ corresponds to an approximation of the force field
\eqref{def:H}.

Note that only the third equation on $\bv^{n+1}_\perp$ is fully implicit
and requires the inversion of a linear operator. Then, from
$\bv^{n+1}$ the first and the second equations give the value for the position $\bx^{n+1}$.

As in the continuous case we are interested in the behavior of the
approximation  $(\bx_\eps, e_{\perp\eps},v_{\mypar\eps})_{\eps>0}$
when $\eps\rightarrow 0$.
 \begin{proposition}[Consistency in the limit $\eps\rightarrow 0$ for a fixed $\Delta t$]
 \label{prop:1}
 Under the assumptions
 (\ref{hyp:2})-(\ref{hyp:b}), we consider a time step $\Delta t>0$, a final time $T>0$  and  the sequence $(\bx^n_\eps,\bv^n_\eps,e_{\perp,\eps}^n)_{0\leq
   n\leq N_T}$  given by (\ref{scheme:0-2}) with $N_T=[T/\Delta t]$, where the initial data
 $(\bx^0_\eps,\bv^0_\eps,e_{\perp,\eps}^0)$ is uniformly bounded with
 respect to $\eps>0$. Then,
\begin{itemize}
\item for all $0\leq n \leq N_T$, $(\bx^n_\eps, \bv^n_\eps,e^n_{\perp,\eps})_{\eps>0}$ is
  uniformly bounded with respect to $\eps>0$ and $\Delta t>0$;

\item for a fixed $\Delta t>0$, the sequence $(\bx^n_\eps,  
  e^n_{\perp,\eps}, v_{\mypar,\eps}^n)_{1\leq n\leq N_T}$ is a
  second order consistent  approximation with respect to $\eps$ to the
 drift-kinetic equation provided by the scheme
\end{itemize} \begin{equation}
\label{sch:y0}
\left\{
\begin{array}{l}
\ds\frac{\bx^{n+1} - \bx^n }{\Delta t} \,\,=\, \bU^{\rm
  gc}(t^n,\bx^n,e_\perp^n, v_\mypar^{n+1}),
\\
\,
\\
\ds\frac{e_{\perp}^{n+1} - e_{\perp}^n }{\Delta t} \,=\, \eps \,e_\perp^n \,{\rm div}_{\bx_\perp} \bF_\perp^\perp(t^n,\bx^n),
\\
\,
\\
\ds\frac{v_{\mypar}^{n+1} - v_{\mypar}^n }{\Delta t} \,\,=\, {E_{\mypar}(t^n,\bx^n)},
\end{array}\right.
\end{equation}
where  $\bF_\perp(t,\bx)$ is defined in (\ref{def:F}) and $\bU^{\rm gc}(t,\bx,e_\perp, v_{\mypar})$ is given by (\ref{def:Ugc}).
\begin{itemize}
\item the scheme (\ref{sch:y0}) is a first order
  approximation in $\Delta t$ of the characteristic curves
  \eqref{eq:characteristic_lim} to (\ref{eq:drift}).
\end{itemize}
 \end{proposition}
\begin{proof}
For clarity reason, we drop the index $\eps$ and set $\lambda_0=b_0\Delta t/\eps$. For any $n\in\{0,\ldots, N_T\}$, we consider
$(\bx^n,\,e^n_{\perp},\,\bv^n)_{\eps>0}$ given by
\eqref{scheme:0-2}.  Thus we first have
$$
\bv^{n+1}_\perp = \left({\rm I}_3 \,-\, \frac{\Delta t}{\eps}\,
  \mR^n\right)^{-1}\left( \bv^{n}_\perp \,+\, \Delta t\,  \bH_\perp(t^n,\bx^n,\bv^n_\perp, e_{\perp}^n)\right), 
$$
where $\bH$ is given by (\ref{def:H}). Then using the assumptions (\ref{hyp:2}) on
$\chi$,  (\ref{hyp:E}) on $\bE$ and (\ref{hyp:b}) on $b$, there exists
a constant $C>0$, such that,
$$
\|\bv^{n+1}_\perp \| \,\leq \; \frac{1}{\sqrt{1+\lambda_0^2}}\,\left( \|\bv^n_\perp\|  \,+\,
  C\,\Delta t\,(1+|e_{\perp}^n|) \right)
$$
and
$$
|e^{n+1}_{\perp}| \,\leq \, |e^{n}_{\perp}|
\,+\, C\,\Delta t \, 
  \|\bv^{n+1}_\perp\|.
$$
Thus, by induction and using that $0<\Delta t < T$, it yields that
there exists another constant $C>0$, independent of $\eps$, 
such that for any $n\in\{0,\ldots, N_T\}$,
$$
\|\bv^{n}_\perp\|\,+\,  |e^{n}_{\perp}| \,\leq\,
\left(\|\bv^{0}_\perp\|\,+\,  |e^{0}_{\perp}| \,+\,
  C\, t^n\right)\,e^{C\,t^n}, 
$$
hence since  the initial data
 $(\bx^0,\bv^0,e_{\perp}^0)_\eps$ is uniformly bounded with
 respect to $\eps>0$, we get for any $n\in\{0,\ldots, N_T\}$, a uniform bound on $(\bv^n)_{\eps>0}$
 and $(e_{\perp}^n)_{\eps>0}$. Therefore,  the uniform bound on the space
 variable $(\bx^n)_{\eps>0}$ also follows. Notice that this bound is
 also uniform with respect to $\Delta t>0$.

Now let us fix $\Delta t >0$. Combining assumption \eqref{hyp:2} with
the bound on $\left(\bv^n\right)_{\eps>0}$, the third equation of \eqref{scheme:0-2}
can be written as
\begin{eqnarray}
\label{flav:0}
\frac{(\bv^{n+1}_\perp)^\perp}{\eps} &=& \frac{1}{b(t^n,\bx_{\perp}^n)}\left(-\frac{\bv^{n+1} - \bv^n }{\Delta t}\,+\, 
\bH(t^n,\bx^n, \bv_\perp^n,e^n_{\perp})\right)_\perp,
\\
\nonumber
&=& \frac{1}{b(t^n,\bx_{\perp}^n)}\left[-\left(\frac{\bv^{n+1} - \bv^n }{\Delta t}\right)_\perp\,+\, \bH_\perp(t^n,\bx^n, \bv_\perp^n,e^n_{\perp}) \right],
\end{eqnarray}
hence it shows that, for any $1\leq n\leq N_T$,
$\left(\eps^{-1}\bv^n_\perp\right)_{\eps>0}$ is uniformly
bounded with respect to $\eps$ (not $\Delta t$) and in
particular $\left(\bv^n_\perp\right)_{\eps>0}$ converges to
zero and
\begin{equation}
\label{bound}
\| \bv^n_\perp \| \,\leq\, C(\Delta t) \,\eps,\quad \forall
n\in\{1,\ldots, N_T\}.
\end{equation}
Therefore we have 
$$
\bv^{n+1}_\perp \,=\,-\frac{\eps}{b(t^n, \bx^n_\perp)}\bH^\perp_\perp(t^n,\bx^n, \bv^n_\perp, e^n_{\perp})  \,+\,    \frac{\eps}{b(t^n, \bx^n_\perp)}\,\left(\frac{\bv^{n+1} - \bv^n }{\Delta t}\right)_\perp^\perp
$$
and substitute it in the first and second equation  of
\eqref{scheme:0-2}. On the one hand, it yields for $n\in\{1,\ldots, N_T\}$,
$$
\frac{\bx_{\perp}^{n+1} - \bx_{\perp}^n }{\Delta t} \,\,=\,
  -\frac{\eps}{b(t^n, \bx^n_\perp)}\left(\bE_\perp(t^n,\bx^n) \,-\, e^n_{\perp}
  \nabla_{\bx_\perp}\ln b(t^n, \bx^n_\perp)\right)^\perp \,+\,
\Theta_1(\bx,\bv,e_{\perp}),
$$
where $\Theta_1$ is given by
$$
\Theta_1(\bx,\bv,e_{\perp}) \,:=\,
\frac{\eps}{b(t^n, \bx^n_\perp)}\,\left(\frac{\bv_\perp^{n+1} -
    \bv^n_\perp }{\Delta t}  +
  \left(\chi(e_{\perp}^n,\bv^n_{\perp}) -e_{\perp}^n\right) \,\nabla_{\bx_\perp}\ln b(t^n, \bx^n_\perp) \right)^\perp. 
 $$
Using that $\chi(e_{\perp}^n,0)=e_{\perp}^n$ and since $\chi
\in W^{1,\infty}_{loc}(\RR\times\RR^3)$ combined with \eqref{bound},
we get that
$$
\|\Theta_1(\bx,\bv,e_{\perp}) \| \,\leq\, C({\Delta t})\,\eps^2. 
$$
On the other hand for the variable $e_{\perp}$, we also have for $n\in\{1,\ldots, N_T\}$,
$$
\frac{e_{\perp}^{n+1} - e_{\perp}^n }{\Delta t} \,=\, 
e_{\perp}^n \, \frac{\eps}{b(t^n, \bx^n_\perp)}\,\langle \bE_\perp(t^n,\bx^n), \nabla_{\bx_\perp}^\perp\ln b(t^n, \bx^n_\perp) \rangle\,+\, \Theta_2(\bx,\bv,e_{\perp}) ,
$$
where $\Theta_2$ is given by $\Theta_2(\bx,\bv,e_{\perp})
:= \langle\bE_\perp(t^n,\bx^n),
\Theta_1(\bx,\bv,e_{\perp})\rangle$, hence it satisfies 
 $$
|\Theta_2(\bx,\bv,e_{\perp})|\leq C({\Delta t}) \eps^2.
$$
Observing that the electric field $\bE$ derives from a
potential $\phi$, we have that  ${\rm div}_\bx \bE_\perp^\perp=0$ and
\begin{eqnarray*}
\frac{\eps}{b(t,\bx_{\perp})}\,\langle
\bE_\perp(t,\bx), \nabla_{\bx_\perp}^\perp\ln
b(t,\bx_{\perp}) \rangle &=& -\frac{\eps}{b^2(t,\bx_\perp)} \,\langle
\bE^\perp_\perp(t,\bx), \nabla_{\bx_\perp}b(t,\bx_{\perp}) \rangle, 
\\
&=& \eps\, {\rm div}_{\bx_\perp} \bF_\perp^\perp(t,\bx).
\end{eqnarray*}
Finally gathering the previous results, we get that for
$n\in\{1,\ldots, N_T\}$, the solution
$(\bx^n,\,e^n_{\perp},\,v^n_{\mypar})_{\eps>0}$ to
\eqref{scheme:0-2}, is a second order approximation to \eqref{sch:y0}
with respect to $\eps$, that is, 
\begin{equation*}
\left\{
\begin{array}{l}
\ds\frac{\bx_{\perp}^{n+1} - \bx_{\perp}^n }{\Delta t} \,\,=\,
  -\frac{\eps}{b(t^n, \bx^n_\perp)}\left(\bE_\perp(t^n,\bx^n) \,-\, e^n_{\perp}
  \nabla_{\bx_\perp}\ln b(t^n, \bx^n_\perp)\right)^\perp \,+\, \Theta_1(\bx,\bv,e_{\perp}),
\\
\,
\\
\ds\frac{x_{\mypar}^{n+1} - x_{\mypar}^n }{\Delta t} \,\,=\, v_{\mypar}^{n+1},
\\
\,
\\
\ds\frac{e_{\perp}^{n+1} - e_{\perp}^n }{\Delta t} \,=\,
  \eps \,e_{\perp}^n \,{\rm div}_{\bx_\perp} \bF_\perp^\perp(t^n,\bx^n)
  \,+ \,\Theta_2(\bx,\bv,e_{\perp}),
\\
\,
\\
\ds\frac{v_{\mypar}^{n+1} - v_{\mypar}^n }{\Delta t} \,\,=\, {E_{\mypar}(t^n,\bx^n)},
\end{array}\right.
\end{equation*}
where 
$$
\|\Theta_1(\bx,\bv,e_{\perp})\|  \,+\,
|\Theta_2(\bx,\bv,e_{\perp})| \,\leq\, C({\Delta t})\,\eps^2.
$$
This proves that $(\bx^n,e^n_{\perp},v_\mypar^n)_{1\leq
  n\leq N_T}$  is a second order consistent approximation
to  the scheme  (\ref{sch:y0}).  

Finally the scheme   (\ref{sch:y0}) is a combination of first order in
$\Delta t$
implicit and explicit Euler scheme, hence the last item is obvious. 
\end{proof}

 \begin{remark}
 The consistency provided by the latter result is far from being uniform with respect to the time step. However though we restrain from doing so here in order to keep
technicalities to a bare minimum, we expect that an analysis similar to the one carried out in
\cite[Section 4]{bibFilbet_Rodrigues} could lead to uniform estimates, proving uniform stability and consistency with respect to $\Delta t$ and $\eps>0$. 
 \end{remark}

  Of course, such a first order scheme is not accurate enough 
  to describe correctly the long time behavior of the numerical solution,
  but it has the advantage of the simplicity. Now, let us see how to generalize such an approach to second and third order schemes. 

 \subsection{Second order semi-implicit Runge-Kutta schemes}
Now, we come to second-order
schemes with two stages. The scheme we consider is a combination of a Runge-Kutta method
for the explicit part and of an L-stable second-order SDIRK method for the implicit part.

 We first choose $\gamma>0$ as the
 smallest root of the polynomial $\gamma^2 - 2\gamma + 1/2 = 0$, {\it
   i.e.} $\gamma = 1 - 1/\sqrt{2}$, then the scheme is given by the
 following two stages. First, we compute an approximation of the
 velocity variable $\bv^{(1)}$ by using an semi-implicit approximation 
 \begin{equation}
 \label{scheme:3-1}
\left\{
\begin{array}{l}
 \ds\frac{\bv^{(1)} \,-\, \bv^n}{\Delta t} \,=\, \gamma\,\bF^{(1)},
 \\ \,\\
\ds\bF^{(1)} \,:=\, \bH(t^n,\bx^n,\bv^n_{\perp},e_\perp^n)  \,-\, b(t^n,\bx_\perp^n)\,\frac{\left(\bv^{(1)}\right)^\perp}{\eps},
\end{array}\right.
 \end{equation}
 where the force term $\bH$ is given by (\ref{def:H}).  For the second stage, we first define  $\hat{t}^{(1)} := t^n + \Delta
 t/{(2\gamma)}$ and by using an explicit procedure we compute 
$\left(\hat{\bx}^{(1)},\hat{\bv}_\perp^{(1)},\hat{e}_\perp^{(1)}\right)$ as
 \begin{equation}
 \label{scheme:3-2}
\left\{
\begin{array}{l}
\ds\frac{\hat{\bx}^{(1)} \,-\, \bx^{n}}{\Delta t}\,=\, \frac{1}{2\gamma}
  \,\bv^{(1)}, \\\,\\ 
\ds\frac{\hat{e}_\perp^{(1)} \,-\, e_\perp^{n}}{\Delta t}\,=\, \frac{1}{2\gamma} \,\langle\bE_\perp(t^n,\bx^n), \bv_\perp^{(1)}\rangle,
\\ \,\\
\ds\frac{\hat{\bv}^{(1)} \,-\, \bv^{n}}{\Delta t}\,=\, \frac{1}{2\gamma} \,\bF^{(1)},
\end{array}\right.
 \end{equation}
 then  the solution $\bv^{n+1}$ is given by
 \begin{equation}
 \label{scheme:3-3}
\left\{
\begin{array}{l}
  \ds\frac{\bv^{n+1} \,-\, \bv^{n}}{\Delta t}  \,=\, (1-\gamma)
   \,\bF^{(1)}  \,+\, \gamma \,\bF^{(2)},
\\ \,\\ 
\ds \bF^{(2)} \,:=\,  \bH(\hat{t}^{(1)}, \hat{\bx}^{(1)},
  \hat{\bv}^{(1)}_\perp, \hat{e}_\perp^{(1)}) \,-\,
  b(\hat{t}^{(1)},
  \hat{\bx}_\perp^{(1)})\,\frac{(\bv^{n+1})^\perp}{\eps}.
\end{array}\right.
 \end{equation}
  Finally, the numerical solution at the new time step is
 \begin{equation}
 \label{scheme:3-4}
 \left\{
 \begin{array}{l}
 \ds\frac{\bx^{n+1} \,-\, \bx^{n}}{\Delta t}  \,=\, (1-\gamma) \,\bv^{(1)}  \,+\, \gamma\,\bv^{n+1},
 \\
 \,
 \\
 \ds \frac{e_\perp^{n+1} \,-\, e_\perp^{n}}{\Delta t}  \,=\, (1-\gamma)\, \langle\bE_\perp(t^n,\bx^n), \bv_\perp^{(1)}\rangle
 \,+\, \gamma \langle\bE_\perp(\hat{t}^{(1)},\hat{\bx}^{(1)}), \bv_\perp^{n+1}\rangle.
  \end{array}\right.
 \end{equation}

 \begin{proposition}[Second order consistency with respect to $\eps$ for a fixed $\Delta t$]
 \label{prop:2}
Under the assumptions
 (\ref{hyp:2})-(\ref{hyp:b}), we consider a time step $\Delta t>0$, a final time $T>0$  and  the sequence $(\bx^n_\eps,\bv^n_\eps,e_{\perp,\eps}^n)_{0\leq
   n\leq N_T}$  given by (\ref{scheme:3-1})-(\ref{scheme:3-4}) with $N_T=[T/\Delta t]$, where the initial data
 $(\bx^0_\eps,\bv^0_\eps,e_{\perp,\eps}^0)$ is uniformly bounded with
 respect to $\eps>0$. Then,
\begin{itemize}
\item for all $0\leq n \leq N_T$, $(\bx^n_\eps, \bv^n_\eps,e^n_{\perp,\eps})_{\eps>0}$ is
  uniformly bounded with respect to $\eps>0$ and $\Delta t>0$;

\item for a fixed $\Delta t>0$, the sequence $(\bx^n_\eps,  
  e^n_{\perp,\eps}, v_{\mypar,\eps}^n)_{1\leq n\leq N_T}$ is a
  second order consistent  approximation with respect to $\eps$ to the
 drift-kinetic equation provided by the scheme 
\end{itemize}
$$
\frac{v_{\mypar}^{(1)} - v_{\mypar}^{n}}{\Delta t}\,=\, \gamma E_\mypar(t^n,\bx^n), 
$$
and
\begin{equation}
 \label{sch:y2-1}
\left\{
\begin{array}{l}
\ds\frac{\hat{\bx}^{(1)} - \bx^{n}}{\Delta t}\,=\, 
  \frac{1}{ 2\gamma}\, \bU^{\rm gc}(t^n,\bx^n,e_{\perp}^n, v_{\mypar}^{(1)}), 
\\\,\\ 
\ds\frac{\hat{e}_{\perp}^{(1)} - e_{\perp}^{n}}{\Delta t}\,=\, \frac{\eps}{ 2\gamma} \,e_{\perp}^{n}\,{\rm div}_{\bx_\perp}\bF_\perp^\perp(t^n,\bx^n), 
\end{array}\right.
\end{equation}
with $\bU^{\rm gc}$ given by \eqref{def:Ugc}, whereas the second stage  $(\bx^{n+1},e_{\perp}^{n+1}, v_\mypar^{n+1})$ is given by
$$
\frac{v_{\mypar}^{n+1} - v_{\mypar}^{n}}{\Delta t}  \,=\, (1-\gamma)\, E_\mypar(t^n,\bx^n)
                                                     \,+\, \gamma\, E_\mypar(\hat{t}^{(1)},\hat{\bx}^{(1)}),
$$
together with 
\begin{equation}
 \label{sch:y2-2}
\left\{
\begin{array}{l}
\ds\frac{\bx^{n+1} - \bx^{n}}{\Delta t}
  \,=\,(1-\gamma)\,\bU^{\rm gc}(t^n,\bx^n,e_{\perp}^n, v_{\mypar}^{(1)})
  \;+\, \gamma\,\bU^{\rm gc}(\hat{t}^{(1)},\hat{\bx}^{(1)},\hat{e}_{\perp}^{(1)}, {v}_{\mypar}^{n+1}),
\\ \,\\
\ds\frac{e_{\perp}^{n+1} - e_{\perp}^{n}}{\Delta t}  \,=\, \eps\left[(1-\gamma)                                               \,e_{\perp}^{n}\,{\rm
                                                  div}_{\bx_\perp}\bF_\perp^\perp(t^n,\bx^n)
                                                     \,+\, \gamma\,\hat{e}_{\perp}^{(1)}\,{\rm
                                                  div}_{\bx_\perp}\bF_\perp^\perp(\hat{t}^{(1)},\hat{\bx}^{(1)})\right].
\end{array}\right.
\end{equation}
\begin{itemize}
\item the scheme (\ref{sch:y2-1})-(\ref{sch:y2-2}) is a second order
  approximation in $\Delta t$ of the characteristic curves
  \eqref{eq:characteristic_lim} to (\ref{eq:drift}).
\end{itemize}
\end{proposition}

 \begin{proof}
 We mainly follow the lines of the proof of
 Proposition~\ref{prop:1}. We set $\lambda_0=b_0\Delta t/\eps$  and for any $n\in\{0,\ldots, N_T\}$, we consider
$(\bx^n,\,e^n_{\perp},\,v^n)_{\eps>0}$ given by
\eqref{scheme:3-1}-\eqref{scheme:3-4}. We first get from \eqref{scheme:3-1},
$$
\|\bv_{\perp}^{(1)}\| \,\leq \; \frac{1}{\sqrt{1+\lambda_0^2}}\,\left( \|\bv_{\perp}^n\|  \,+\,
  C\,\Delta t\,(1+|e_{\perp}^n|) \right),
$$ 
whereas from \eqref{scheme:3-3}, we have
$$
\|\bv_{\perp}^{n+1}\| \,\leq \; \frac{1}{\sqrt{1+\lambda_0^2}}\,\left[ \frac{2\gamma-1}{\gamma}\|\bv_{\perp}^{n}\| + \frac{1-\gamma}{\gamma}\|\bv_{\perp}^{(1)}\|  \,+\,
  C\,\Delta t\,(1+|\hat{e}_{\perp}^{(1)}|) \right].
$$
Furthermore, using (\ref{scheme:3-2}) on  $\hat{e}_{\perp}^{(1)}$
and the last stage \eqref{scheme:3-4} on $e_{\perp}^{n+1}$, we obtain
$$
\left\{
\begin{array}{l}
|e^{(1)}_{\perp}| \,\leq \, |e^{n}_{\perp}|
\,+\, C\,\Delta t \,\|\bv^{(1)}_{\perp}\|,
\\ \, \\
|e^{n+1}_{\perp}| \,\leq \, |e^{n}_{\perp}|
\,+\, C\,\Delta t \, 
  \left( \|\bv^{(1)}_{\perp}\| \,+\, \|\bv^{n+1}_{\perp}\|  \right).
\end{array}\right.
$$
Thus, by induction and using that $0<\Delta t < T$, it yields that
there exists another constant $C>0$, independent of $\eps$, 
such that for any $n\in\{0,\ldots, N_T\}$,
$$
\|\bv^{n}_{\perp}\|\,+\,  |e^{n}_{\perp}| \,\leq\,
\left(\|\bv^{0}_{\perp}\|\,+\,  |e^{0}_{\perp}| \,+\,
  C\, t^n\right)\,e^{C\,t^n}, 
$$
hence since  the initial data
 $(\bx^0,\bv^0,e_{\perp}^0)_{\eps>0}$ is uniformly bounded with
 respect to $\eps>0$, we get for any $n\in\{0,\ldots, N_T\}$, a uniform bound on $(\bv^n)_{\eps>0}$
 and $(e_{\perp}^n)_{\eps>0}$, then also on $(\bx^n)_{\eps>0}$.

Now we fix $\Delta t >0$ and combining assumption \eqref{hyp:2} with
the bound on $\left(\bv^n\right)_{\eps>0}$, the first stage
\eqref{scheme:3-1}  and can be written as
$$
\frac{(\bv^{(1)}_\perp)^\perp}{\eps} = \frac{1}{b(t^n, \bx^n_\perp)}\left(-\frac{\bv^{n+1}_{\perp} - \bv^n_{\perp} }{\Delta t}\,+\, 
\bH_\perp(t^n,\bx^n,\bv_{\perp}^n,e_{\perp}^n) \right),
$$
hence, $\left(\eps^{-1}\bv^{(1)}_{\perp}\right)_{\eps>0}$ is uniformly
bounded with respect to $\eps$ (not $\Delta t$). Furthermore, we
observe that $(\hat{\bv}^{(1)}, \hat{e}_\perp^{(1)})$ are also bounded
since $\bv^{(1)}$ is bounded and $\hat{\bv}^{(1)}$ a combination of $\bv^n$ and
$\bv^{(1)}$. Then  the second stage
\eqref{scheme:3-3}  also gives that
$$
\frac{(\bv^{n+1}_\perp)^\perp}{\eps} = \frac{1}{b(\hat{t}^{(1)},\hat{\bx}_{\perp}^{(1)})}\,\left(- \frac{\bv_{\perp}^{n+1}-\alpha \,\bv_{\perp}^{(1)} + (1-\alpha)\,\bv^n_{\perp}}{\Delta t} \,+\,\bH_\perp(\hat{t}^{(1)}, \hat{\bx}^{(1)},
  \hat{\bv}_{\perp}^{(1)}, \hat{e}_{\perp}^{(1)})  \right),
$$ 
with $\alpha=(1-\gamma)/\gamma$, thus $\left(\eps^{-1}\bv^{n+1}_{\perp}\right)_{\eps>0}$ is also uniformly
bounded with respect to $\eps$ (not $\Delta t$). In particular $\left(\bv^n_{\perp}\right)_{\eps>0}$ converges to
zero and
\begin{equation}
\label{bound:2}
\| \bv^n_{\perp}\| \,+\, \| \bv^{(1)}_{\perp}\| \,\leq\, C \,\eps,\quad \forall
n\in\{1,\ldots, N_T\}.
\end{equation}
Now we can substitute $\bv_\perp^{(1)}$ and $\bv_\perp^{n+1}$ in
\eqref{scheme:3-2} and \eqref{scheme:3-4} and get 
$$
\left\{
\begin{array}{l}
\ds\frac{\hat{\bx}^{(1)}_\perp - \bx^{n}_\perp}{\Delta t}\,=\, 
  \frac{1}{ 2\gamma}\, \bU_\perp^{\rm gc}(t^n,\bx^n,e_{\perp}^n) \,+\, \Theta_1(\bx,\bv,e_{\perp}), 
\\\,\\ 
\ds\frac{\hat{e}_{\perp}^{(1)} - e_{\perp}^{n}}{\Delta t}\,=\, \frac{\eps}{ 2\gamma} \,e_{\perp}^{n}\,{\rm div}_{\bx_\perp}\bF_\perp^\perp(t^n,\bx^n) \,+\, \Theta_2(\bx,\bv,e_{\perp}), 
\end{array}\right.
$$ 
with
$$
\|\Theta_1(\bx,\bv,e_{\perp})\| \,+\,
|\Theta_2(\bx,\bv,e_{\perp})| \,\leq\,C({\Delta t})\,\eps^2,
$$ 
whereas in the parallel direction the scheme remains the
same, 
$$
\left\{
\begin{array}{l}
\ds\frac{\hat{x}_\mypar^{(1)} - \hat{x}_\mypar^{n}}{\Delta t}\,=\, \frac{1}{ 2\gamma}\, v_\mypar^{(1)},
\\ \, \\
\ds\frac{v_\mypar^{(1)} \,-\, v_\mypar^n}{\Delta t} \,=\, \gamma\,E_\mypar(t^n,\bx^n).
\end{array}\right.
$$
Then we define  
$\hat{\bx}^{(1)}=(\hat{\bx}_\perp^{(1)},\hat{x}_\mypar^{(1)})$ and
apply the
second stage, $(\bx_{\perp}^{n+1}, e_{\perp}^{n+1})$ is given by
$$
\left\{
\begin{array}{l}
\ds\frac{\bx_{\perp}^{n+1} - \bx^{n}_{\perp}}{\Delta t}
  \,=\,(1-\gamma)\,\bU^{\rm gc}_\perp(t^n,\bx^n,e_{\perp}^n)
  \;+\, \gamma\,\bU^{\rm gc}_\perp(\hat{t}^{(1)},\hat{\bx}^{(1)},\hat{e}_{\perp}^{(1)})
  + \Theta_3(\bx,\bv,e_{\perp}),
\\ \,\\
\ds\frac{e_{\perp}^{n+1} - e_{\perp}^{n}}{\Delta t}  \,=\, \eps\left[(1-\gamma)
                                               \,e_{\perp}^{n}\,{\rm
                                                  div}_{\bx_\perp}\bF_\perp^\perp(t^n,\bx^n)
                                                     \,+\, \gamma\,\hat{e}_{\perp}^{(1)}\,{\rm
                                                  div}_{\bx_\perp}\bF_\perp^\perp(\hat{t}^{(1)},\hat{\bx}^{(1)})\right] \,+\, \Theta_4(\bx,\bv,e_{\perp}),
\end{array}\right.
$$
where 
$$
\|\Theta_3(\bx_\eps,\bv_\eps,e_{\perp,\eps})\| \,+\,
|\Theta_4(\bx_\eps,\bv_\eps,e_{\perp,\eps})| \,\leq\,C({\Delta t})\,\eps^2,
$$ 
whereas in the parallel direction we have no consistency error 
$$
\left\{
\begin{array}{l}
\ds\frac{x_\mypar^{n+1} \,-\, x_\mypar^n}{\Delta t} \,=\,
  (1-\gamma)\,v_\mypar^{(1)} + \gamma\,v_\mypar^{n+1},
\\ \, \\
\ds\frac{v_\mypar^{n+1} \,-\, v_\mypar^n}{\Delta t} \,=\,
  (1-\gamma)\,E_\mypar(t^n,\bx^n) + \gamma\,E_\mypar(\hat{t}^{(1)},\hat{\bx}^{(1)}).
\end{array}\right.
$$

To prove the last item we observe that the scheme
(\ref{sch:y2-1})-(\ref{sch:y2-2}) is the stiffly accurate
semi-implicit Runge-Kutta scheme \cite{BFR:15}, which is a second
order explicit scheme for the variable $(\bx_\perp, e_\perp,
v_\mypar)$ and implicit for $x_\mypar$. This approximation is then 
second order  with the characteristic curves
  \eqref{eq:characteristic_lim}.
\end{proof}

The present scheme is $L$- stable, which means uniformly linearly stable with
respect to $\Delta t$.

\subsection{Third order semi-implicit Runge-Kutta schemes}
A  third order semi-implicit scheme is given by a four stages
Runge-Kutta method introduced in the framework of hyperbolic
systems with stiff source terms \cite{BFR:15}. First,  we set $\alpha=0.24169426078821$, $\beta =
\alpha/4$ and $\eta= 0.12915286960590$ and $\gamma=1/2-\alpha-\beta-\eta$. Then we construct the first
stage as 
 \begin{equation}
 \label{scheme:4-1}
\left\{
\begin{array}{l}
\ds\frac{\bv^{(1)} \,-\, \bv^n}{\Delta t} \,= \, \alpha\,\bF^{(1)},
\\ \,\\
\ds\bF^{(1)} \,:=\, \bH(t^n,\bx^n,\bv_\perp^n,e_\perp^n)  \,-\, b(t^n,\bx_\perp^n)\,\frac{\left(\bv^{(1)}\right)^\perp}{\eps},
\end{array}\right.
 \end{equation}
with $\bH$ given by (\ref{def:H}). For the second stage, we have
\begin{equation}
\label{scheme:4-2}
\left\{
\begin{array}{l}
\ds\frac{\bv^{(2)} \,-\, \bv^{n}}{\Delta t} \,=\,   -\alpha\,\bF^{(1)} \,+\, \alpha \,\bF^{(2)},
\\ \,\\
\ds \bF^{(2)} \,:=\,  \bH(t^n,\bx^n,\bv_\perp^n,e_\perp^n)  \,-\,
b(t^n,\bx_\perp^n)\,\frac{\left(\bv^{(2)}\right)^\perp}{\eps}.
\end{array}\right.
\end{equation}
Then, for the third stage we set 
\begin{equation}
\label{scheme:4-3}
\left\{
\begin{array}{l}
\ds\frac{\hat{\bx}^{(2)} \,-\, \bx^{n}}{\Delta t}\,= \,\bv^{(2)},
\\ \,\\
\ds\frac{\hat{e}_\perp^{(2)} \,-\, e_\perp^{n}}{\Delta t} \,= \,\langle\bE_\perp(t^n,\bx^n),\bv_\perp^{(2)}\rangle,
\\ \,\\
\ds\frac{\hat{\bv}^{(2)} \,-\, \bv^{n}}{\Delta t} \,= \,\bF^{(2)},
\end{array}\right.
\end{equation}
and we compute a new approximation $\bv^{(3)}$ as
\begin{equation}
\label{scheme:4-4}
\left\{
\begin{array}{l}
\ds\frac{\bv^{(3)} \,-\, \bv^{n}}{\Delta t} \,=\,  (1-\alpha)
  \,\bF^{(2)} \,+\, \alpha \,\bF^{(3)},
\\ \, \\
\ds\bF^{(3)} \,:=\,  \bH(t^{n+1}, \hat{\bx}^{(2)},
  \hat{\bv}^{(2)}_\perp, \hat{e}_\perp^{(2)}) \,-\,
  b(t^{n+1},
  \hat{\bx}_\perp^{(2)})\,\frac{(\bv^{(3)})^\perp}{\eps}.
\end{array}\right.
\end{equation}
Finally, for the fourth stage we set 
\begin{equation}
\label{scheme:4-5}
\left\{
\begin{array}{l}
\ds\frac{\hat{\bx}^{(3)} \,-\, \bx^{n}}{\Delta t}\,=\, \frac{1}{4}
  \left(\bv^{(2)} +\bv^{(3)} \right),
\\ \,\\
\ds\frac{\hat{e}_\perp^{(3)} \,-\, e_\perp^{n}}{\Delta t} \,=\,  \frac{1}{4}
  \,\left(\langle\bE_\perp(t^n,\bx^n),\bv_\perp^{(2)}\rangle \,+\,  \langle\bE_\perp(t^{n+1},\hat{\bx}^{(2)}),\bv_\perp^{(3)}\rangle \right),
\\ \,\\
\ds\frac{\hat{\bv}^{(3)} \,-\, \bv^{n}}{\Delta t} \,=\,  \frac{1}{4} \,\left(\bF^{(2)}\,+\,\bF^{(3)}\right)
\end{array}\right.
\end{equation}
and we compute a new approximation $\bv^{(4)}$ as 
\begin{equation}
\label{scheme:4-6}
\left\{
\begin{array}{l}
\ds\frac{\bv^{(4)} \,-\, \bv^{n}}{\Delta t} \,=\, \beta \,{\bF}^{(1)}\,+\, \eta \,{\bF}^{(2)}\,+\,
   \gamma \,{\bF}^{(3)}\,+\,  \alpha
    \,{\bF}^{(4)},
\\ \, \\
\ds\bF^{(4)} \,:=\, \bH(t^{n+1/2}, \hat{\bx}^{(3)},
  \hat{\bv}^{(3)}_\perp, \hat{e}_\perp^{(3)}) \,-\,
  b(t^{n+1},
  \hat{\bx}_\perp^{(3)})\,\frac{(\bv^{(4)})^\perp}{\eps},
\end{array}\right.
\end{equation}
where $t^{n+1/2}=t^n + \Delta t/2$.
Finally, the numerical solution at the new time step is
\begin{equation}
\label{scheme:4-7}
\left\{
\begin{array}{l}
\ds\frac{\bx^{n+1} \,-\, \bx^{n}}{\Delta t}  \,=\, \frac{1}{6} \left( \bv^{(2)} \,+\,
  \bv^{(3)}  \,+\, 4\, \bv^{(4)} \right),
\\
\,
\\
\ds \frac{e_\perp^{n+1} \,-\, e_\perp^{n}}{\Delta t}  \,=\, \frac{1}{6} \left( \langle\bE_\perp(t^n,\bx^n),\bv_\perp^{(2)}\rangle \,+\,
  \langle\bE_\perp(t^{n+1},\hat{\bx}^{(2)}),\bv_\perp^{(3)}\rangle  \,+\, 4\, \langle\bE_\perp(t^{n+1/2},\hat{\bx}^{(3)}),\bv_\perp^{(4)}\rangle \right),
\\
\,
\\
\ds\frac{\bv^{n+1} \,-\, \bv^{n}}{\Delta t}  \,=\, \frac{1}{6} \left( {\bf F}^{(2)} \,+\,
  {\bf F}^{(3)}  \,+\, 4 \,{\bf F}^{(4)} \right).
\end{array}\right.
\end{equation}
As for the previous schemes, under uniform stability assumptions with
respect to $\eps>0$, we prove the following Proposition

\begin{proposition}[Second order consistency with respect to $\eps$  for a fixed $\Delta t$]
\label{prop:3}
Under the assumptions
 (\ref{hyp:2})-(\ref{hyp:b}), we consider a time step $\Delta t>0$, a final time $T>0$  and  the sequence $(\bx^n_\eps,\bv^n_\eps,e_{\perp,\eps}^n)_{0\leq
   n\leq N_T}$  given by (\ref{scheme:4-1})-(\ref{scheme:4-7}) with $N_T=[T/\Delta t]$, where the initial data
 $(\bx^0_\eps,\bv^0_\eps,e_{\perp,\eps}^0)$ is uniformly bounded with
 respect to $\eps>0$. Then,
\begin{itemize}
\item for all $0\leq n \leq N_T$, $(\bx^n_\eps, \bv^n_\eps,e^n_{\perp,\eps})_{\eps>0}$ is
  uniformly bounded with respect to $\eps>0$ and $\Delta t>0$;

\item for a fixed $\Delta t>0$, the sequence $(\bx^n_\eps,  
  e^n_{\perp,\eps}, v_{\mypar,\eps}^n)_{1\leq n\leq N_T}$ is a
  second order consistent  approximation with respect to $\eps$ to the
 drift-kinetic equation provided by the scheme $v_{\mypar}^{(2)} =
 v_{\mypar}^{n}$, with
\end{itemize}
\begin{equation}
 \label{sch:y4-1}
\left\{
\begin{array}{l}
\ds\frac{\hat{\bx}^{(2)} \,-\, \bx^{n}}{\Delta t}\,= \, \bU^{\rm gc}(t^n,\bx^n,e_{\perp}^n, v_{\mypar}^{(2)}),
\\ \,\\
\ds\frac{\hat{e}_\perp^{(2)} \,-\, e_\perp^{n}}{\Delta t} \,= \,  \eps \,e_{\perp}^{n}\,{\rm div}_{\bx_\perp}\bF_\perp^\perp(t^n,\bx^n). 
\end{array}\right.
\end{equation}
The next stage is given by
\begin{equation*}
 %\label{sch:y4-2}
\frac{v_{\mypar}^{(3)} - v_{\mypar}^{n}}{\Delta t}\,=\, (1-\alpha)\,E_\mypar(t^n,\bx^n)+\alpha\,E_\mypar(t^{n+1/2},\hat{\bx}^{(2)}), 
\end{equation*}
and
\begin{equation}
 \label{sch:y4-3}
\left\{
\begin{array}{l}
\ds\frac{\hat{\bx}^{(3)} \,-\, \bx^{n}}{\Delta t}\,= \, \frac{1}{4}\left[\bU^{\rm gc}(t^n,\bx^n,e_{\perp}^n, v_{\mypar}^{(2)})\,+\; \bU^{\rm gc}(t^{n+1/2},\hat{\bx}^{(2)},\hat{e}_{\perp}^{(2)}, v_{\mypar}^{(3)})\right],
\\ \,\\
\ds\frac{\hat{e}_\perp^{(3)} \,-\, e_\perp^{n}}{\Delta t} \,= \,  \frac{\eps}{4}\left[
  \,e_{\perp}^{n}\,{\rm div}_{\bx_\perp}\bF_\perp^\perp(t^n,\bx^n) \,+\, \hat{e}_{\perp}^{(2)}\,{\rm div}_{\bx_\perp}\bF_\perp^\perp(t^{n+1/2},\hat{\bx}^{(2)}) \right]. 
\end{array}\right.
\end{equation}
The fourth stage is
\begin{equation*}
 %\label{sch:y4-3}
\frac{v_{\mypar}^{(4)} - v_{\mypar}^{n}}{\Delta t}\,=\,
  (\beta+\eta)\,E_\mypar(t^n,\bx^n)+\gamma\,E_\mypar(t^{n+1/2},\hat{\bx}^{(2)}),
\end{equation*}
and finally
\begin{equation}
 \label{sch:y4-4}
\left\{
\begin{array}{l}
\ds\frac{\bx^{n+1} - \bx^{n}}{\Delta t}
  \,=\, \bU^{\rm gc}_f,
\\ \,\\
\ds\frac{e_{\perp}^{n+1} - e_{\perp}^{n}}{\Delta t}  \,=\, u^\mypar_f,
\\ \,\\
\ds\frac{v_{\mypar}^{n+1} - v_{\mypar}^{n}}{\Delta t}  \,=\, \frac{1}{6}\left[ E_\mypar(t^n,\bx^n)
                                                     \,+\,
  E_\mypar(\hat{t}^{(2)},\hat{\bx}^{(2)}) \,+\,   4\,E_\mypar(\hat{t}^{(3)},\hat{\bx}^{(3)})\right],
\end{array}\right.
\end{equation}
with $\bU_f^{\rm gc}$ and $u^\mypar_f$ given by
$$
\left\{
\begin{array}{l}
\ds\bU^{\rm gc}_f \,:=\, \frac{1}{6}\left[\bU^{\rm gc}(t^n,\bx^n,e_{\perp}^n, v_{\mypar}^{(2)})
  \;+\, \bU^{\rm gc}(t^{n+1/2},\hat{\bx}^{(2)},\hat{e}_{\perp}^{(2)}, {v}_{\mypar}^{(3)}) \;+\, 4\bU^{\rm gc}(t^{n+1},\hat{\bx}^{(3)},\hat{e}_{\perp}^{(3)}, {v}_{\mypar}^{(4)})\right],
\\ \,\\
\ds u^\mypar_f\,:=\,\frac{\eps}{6}\left[e_{\perp}^{n}\,{\rm
                                                  div}_{\bx_\perp}\bF_\perp^\perp(t^n,\bx^n)
                                                     \,+\, \hat{e}_{\perp}^{(2)}\,{\rm
                                                  div}_{\bx_\perp}\bF_\perp^\perp(t^{n+1/2},\hat{\bx}^{(2)}) \,+\, 4\,\hat{e}_{\perp}^{(3)}\,{\rm
                                                  div}_{\bx_\perp}\bF_\perp^\perp(t^{n+1},\hat{\bx}^{(3)})\right];
\end{array}\right.
$$
\begin{itemize}
\item the scheme (\ref{sch:y4-1})-(\ref{sch:y4-3}) is a third order
  approximation in $\Delta t$ of the characteristic curves
  \eqref{eq:characteristic_lim} to (\ref{eq:drift}).
\end{itemize}
\end{proposition}
We skip the proof of Proposition~\ref{prop:3} since it follows the same arguments as in  Proposition~\ref{prop:1} and  Proposition~\ref{prop:2}.

%%%%%%%%%%%%%%%%%%%%%%%%%%%%%%%%%%%%%%%%%%%%%%%%%%%%%%%%%%%%%%%%%%%%%%

\section{Discretization of the Poisson  equation}
\label{sec:4}
\setcounter{equation}{0}
{Thanks to periodic boundary condition in $z$-direction, the 3D Poisson equation~\eqref{eq:poisson}-\eqref{eq:DK_BC1} can be decomposed into a series of 2D Poisson equations by applying Fourier transform.
Then we use  a classical five points finite difference approximation to
discretize  the 2D Poisson equations.}  So it remains to treat
the Dirichlet boundary conditions on $\partial D$. 

Johansen {\it et al.}~\cite{bibJohansen} have proposed an embedded boundary approach for the Poisson's equation, which uses a
finite-volume discretization which embeds the domain in a regular Cartesian grid.
It provides  a conservative discretization for   engineering problems, such
as viscous fluid flow or heat conduction, on irregular domains. 
However, for the Vlasov-Poisson system, a classical finite difference method is usually used and is proven to be efficient and accurate~\cite{bibFS, bibGV, bibUtsumi}. By following this direction, we thus propose a finite difference discretization which embeds the domain in a regular Cartesian grid.

To discretize the Laplacian operator $\Delta_{\mathbf{x}_\perp}\phi$ near the physical boundary, some points of the usual five points finite difference formula can be located outside of interior domain. 
For instance,   Figure~\ref{fig:2Ddomain} illustrates the
discretization stencil for  $\Delta_{\mathbf{x}_\perp}\phi$ at the point $(x_i,y_j)$.
We notice that the point $\mathbf{x}_g=(x_i,y_{j-1})$ is located outside of interior domain. Let us denote the approximation of $\phi$ at  the point $\mathbf{x}_g$ by $\phi_{i,j-1}$.
Thus $\phi_{i,j-1}$ should be extrapolated from the interior domain.

We extrapolate $\phi_{i,j-1}$ on the normal direction $\mathbf{n}$ 
\begin{equation}
 \phi_{i,j-1}\,=\,\bar{w}_p \,{\phi}(\mathbf{x}_p) \,+\, \bar{w}_h \,{\phi}(\mathbf{x}_h) \,+\, \bar{w}_{2h}\, {\phi}(\mathbf{x}_{2h}),
 \label{eq:normal_extrapolation}
\end{equation}
where $\mathbf{x}_p$ is the cross point of the normal $\mathbf{n}$ and the physical boundary $D$.
The points $\mathbf{x}_h$ and $\mathbf{x}_{2h}$ are  equal spacing on the normal $\mathbf{n}$, {\it i.e.} $h=|\mathbf{x}_p-\mathbf{x}_h|=|\mathbf{x}_h-\mathbf{x}_{2h}|$,
with $h=\min(\Delta x,\Delta y)$, $\Delta x$, $\Delta y$ are the space steps in the directions $x$ and $y$ respectively.
Moreover, $\bar{w}_p$, $\bar{w}_h$, $\bar{w}_{2h}$ are the extrapolation weights depending on the position of $\mathbf{x}_g$, $\mathbf{x}_p$, $\mathbf{x}_h$ and $\mathbf{x}_{2h}$.
In~\eqref{eq:normal_extrapolation}, ${\phi}(\mathbf{x}_p)$ is given by the  boundary condition~\eqref{eq:DK_BC1},
whereas ${\phi}(\mathbf{x}_h)$, ${\phi}(\mathbf{x}_{2h})$ should be determined by interpolation.
\begin{figure}[h]
  \begin{center} 
    \includegraphics[width=10cm]{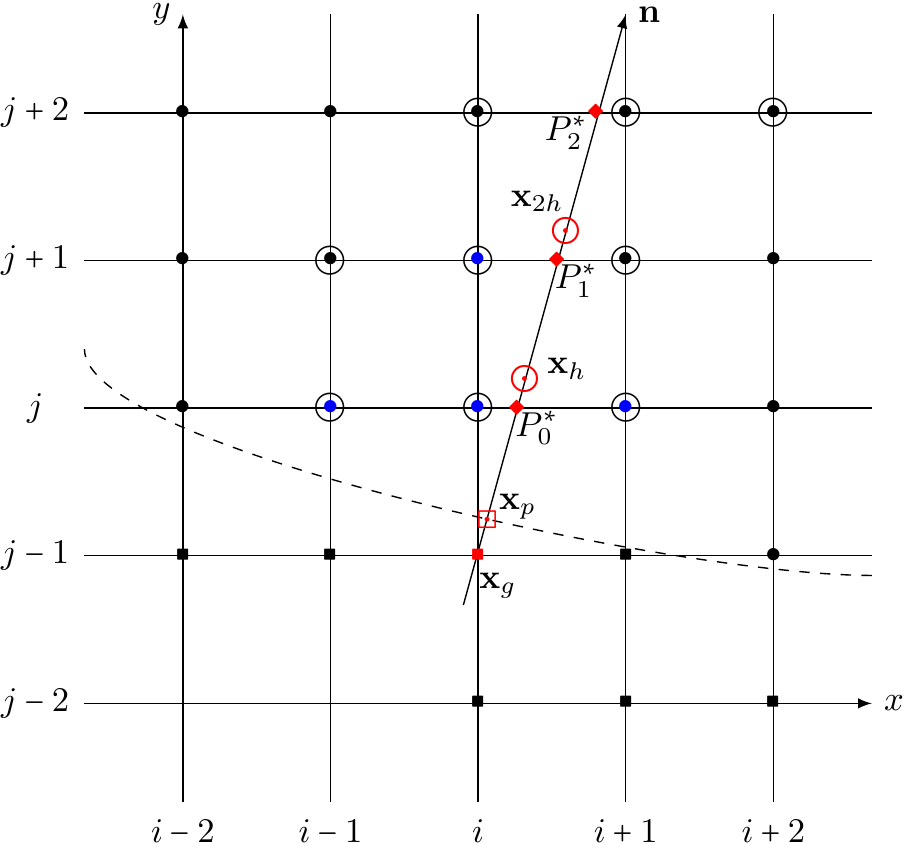}
\caption{\label{fig:2Ddomain}Spatially two-dimensional Cartesian mesh. $\bullet$ is interior point, $\filledsquare$ is ghost point, $\boxdot$ is the point at the boundary, $\largecircle$ is the point for extrapolation, the dashed line is the boundary.}
  \end{center}
\end{figure}

For this, we first construct an interpolation stencil $\mathcal{E}$, composed of grid points of $D$. 
For instance, in Figure~\ref{fig:2Ddomain},  the inward normal $\mathbf{n}$ intersects  the grid lines $y=y_{j}$, $y_{j+1}$, $y_{j+2}$ at points $P^*_0$, $P^*_1$, $P^*_2$. 
Then we choose the three nearest points  of the cross point $P^*_l,\,l=0,1,2$, in each line, {\it i.e.} marked by a large circle. 
From these nine points, we can build a  Lagrange polynomial $q_2(\mathbf{x})\in\mathbb{Q}_2(\mathbb{R}^2)$.
Therefore,  we evaluate the polynomial $q_2(\mathbf{x})$ at
$\mathbf{x}_h$ and $\mathbf{x}_{2h}$, {\it i. e.}
$$\left\{
\begin{array}{l}
 \ds{\phi}(\mathbf{x}_h)\,=\,\sum_{\ell=0}^8w_{h,\ell}\,{\phi}(\mathbf{x}_{\ell}),
 \\ \,
\\
\ds{\phi}(\mathbf{x}_{2h})\,=\,\sum_{\ell=0}^8 w_{2h,\ell}\,{\phi}(\mathbf{x}_{\ell}),
\end{array}\right.
$$
with $\mathbf{x}_{\ell}\in \mathcal{E}$. 
We thus have that $\phi_{i,j-1}$ is approximated from the interior domain.

However, in some cases, we can not find a stencil of nine interior points. 
For instance, when the interior domain has small acute angle sharp, the normal $\mathbf{n}$ can not have three cross points $P^*_l,\,l=0,1,2$ in interior domain, or we can not have three nearest points  of the cross point $P^*_l,\,l=0,1,2$, in each line. 
In this case, we alternatively use a first degree polynomial $q_1(\mathbf{x})$ with a four points  stencil or even a zero degree polynomial $q_0(\mathbf{x})$ with  an one point  stencil. 
We can similarly construct the four points stencil or the one point stencil as above.

 \section{Numerical simulations}
 \label{sec:5}
\setcounter{equation}{0}
 
\subsection{One single particle motion in $3D$}
Before simulating at the statistical level, we investigate on the motion of individual particles in a given magnetic field the accuracy and stability properties with respect to $\eps>0$ of the semi-implicit algorithms presented in Section~\ref{sec:3}.  

Numerical experiments of the present subsection are run with an
electric field $\bE=-\nabla_\bx \phi$, with
$$
\phi\,:\quad \RR^3\to\RR\,,\qquad
\bx=(x,y,z)\,\mapsto\,20\,\sqrt{x^2+y^2} + 0.5\,\cos(2\pi\,z) 
$$
and a time-independent external magnetic field corresponding to 
$$
b\,:\quad \RR^3\to\RR\,,\qquad \bx_\perp=(x,y,0)\,\mapsto\,\frac{1}{10^2\,-\,(x^2+y^2)}
$$
Moreover we choose for all simulations the initial data as
$\bx^0=(5,0,0)$, $\bv^0=(4,3,2)$, whereas the final time is $T=10$. In
this case,  the asymptotic drift velocity predicted by the limiting
model~\eqref{eq:characteristic_lim} is explicitly given by  $\bU^{\rm
  gc}$.

First, for comparison, we compute a reference solution $(\bx^\eps,\bv^\eps,e_\perp^\eps)_{\eps>0}$
to the initial problem \eqref{eq:characteristic_scaled_vlasov2} thanks to an explicit
fourth-order Runge-Kutta scheme used with a very small time step of
the order of $\eps$ and a reference solution $(\bx,e_\perp,v_\mypar)$ to the (non
stiff) asymptotic model \eqref{eq:characteristic_lim} obtained when $\eps
\ll 1$. Recall that the derivation of \eqref{eq:characteristic_lim} also
shows that $\bv^\eps$ is second order consistent to $\bU^{\rm gc}\equiv\bU^{\rm gc}(\bx,e_\perp,v_\mypar)$ when $\eps\ll 1$. Then we compute an approximate solution
$(\bx^\eps_{\Delta t}, \bv^\eps_{\Delta t}, e^\eps_{\perp\Delta t})$  using either \eqref{scheme:3-1}--\eqref{scheme:3-4} or \eqref{scheme:4-1}-\eqref{scheme:4-7}, and compare them to the reference solutions. 

In Figures~\ref{fig0:1} and \ref{fig0:2}, we present trajectory on space variables between the reference solution for the initial
problem \eqref{eq:characteristic_scaled_vlasov2} and the one  obtained
with the third-order scheme \eqref{scheme:4-1}-\eqref{scheme:4-7}. As
expected for a fixed time step $\Delta t=0.1$, the scheme is quite stable even in the limit $\eps\ll 1$
and the error on the space variable is uniformly small. In contrast, for a
fixed time step, the error on the velocity variable is small  for large
values of $\eps$, but gets very large when $\eps\ll 1$ since the
scheme cannot follow high-frequency time oscillations of order
$\eps^{-1}$ when $\eps\ll \Delta t$ (not presented).
 
\begin{figure}
\begin{center}
 \begin{tabular}{cc}
\includegraphics[width=7.cm]{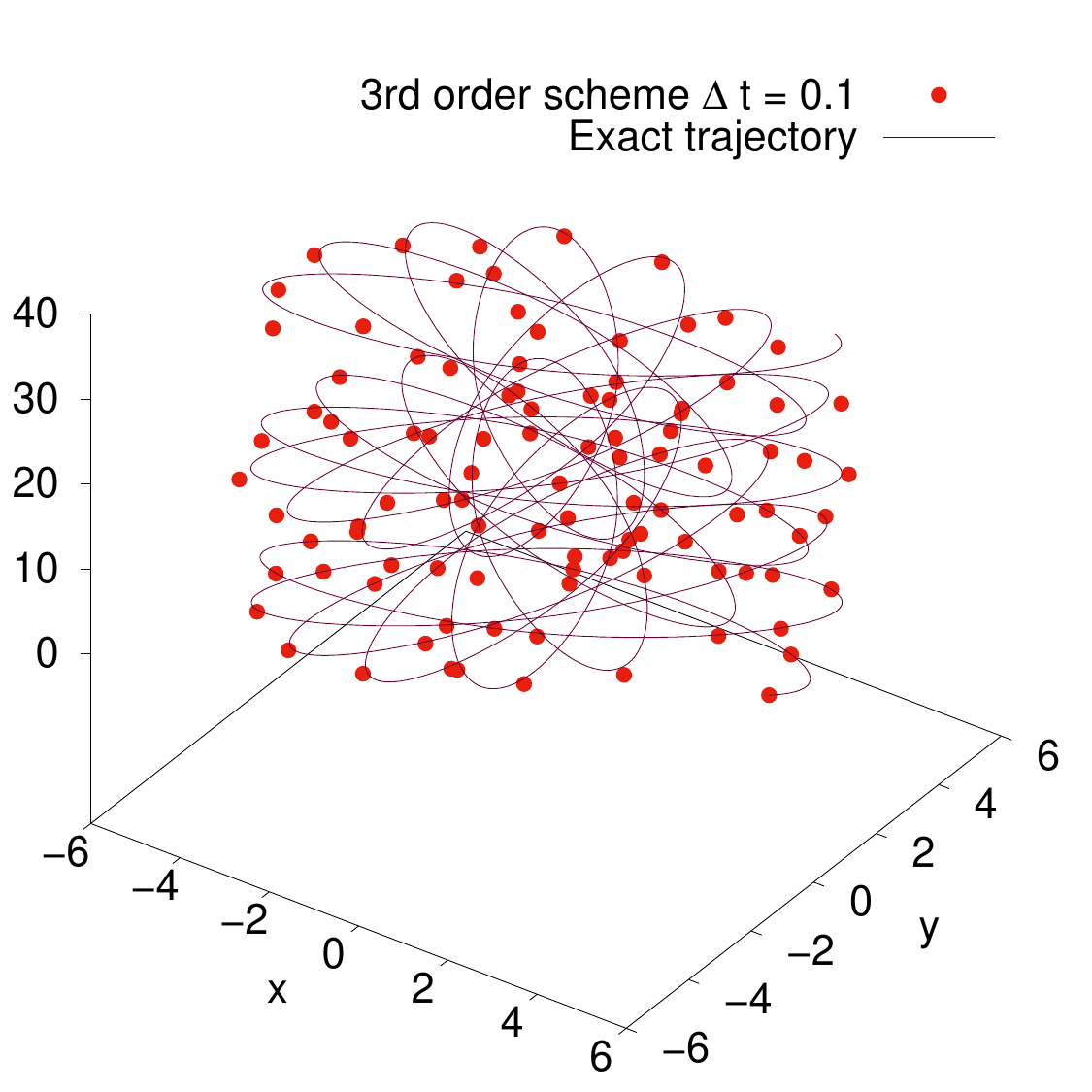} &    
\includegraphics[width=7.cm]{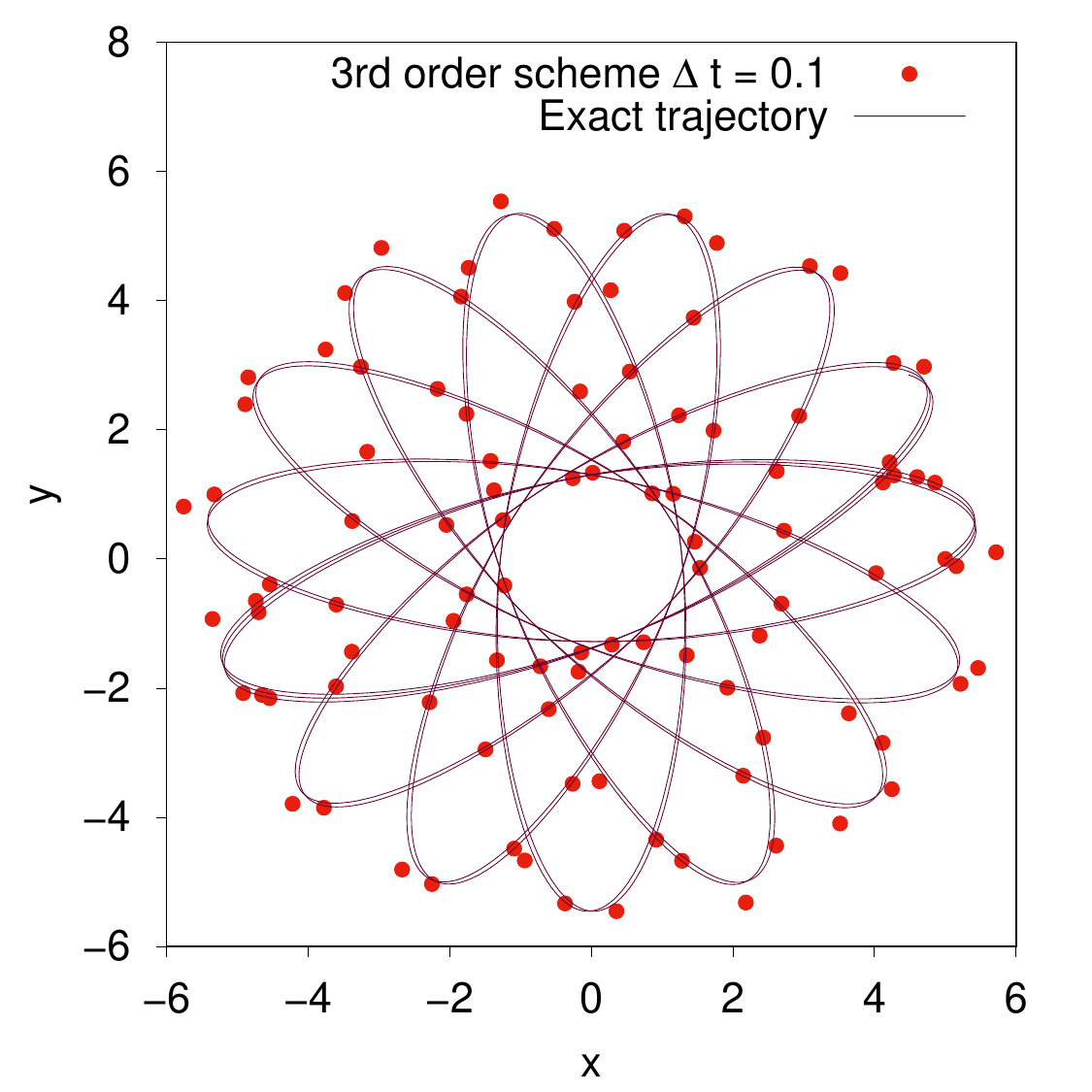} 
\\
\includegraphics[width=7.cm]{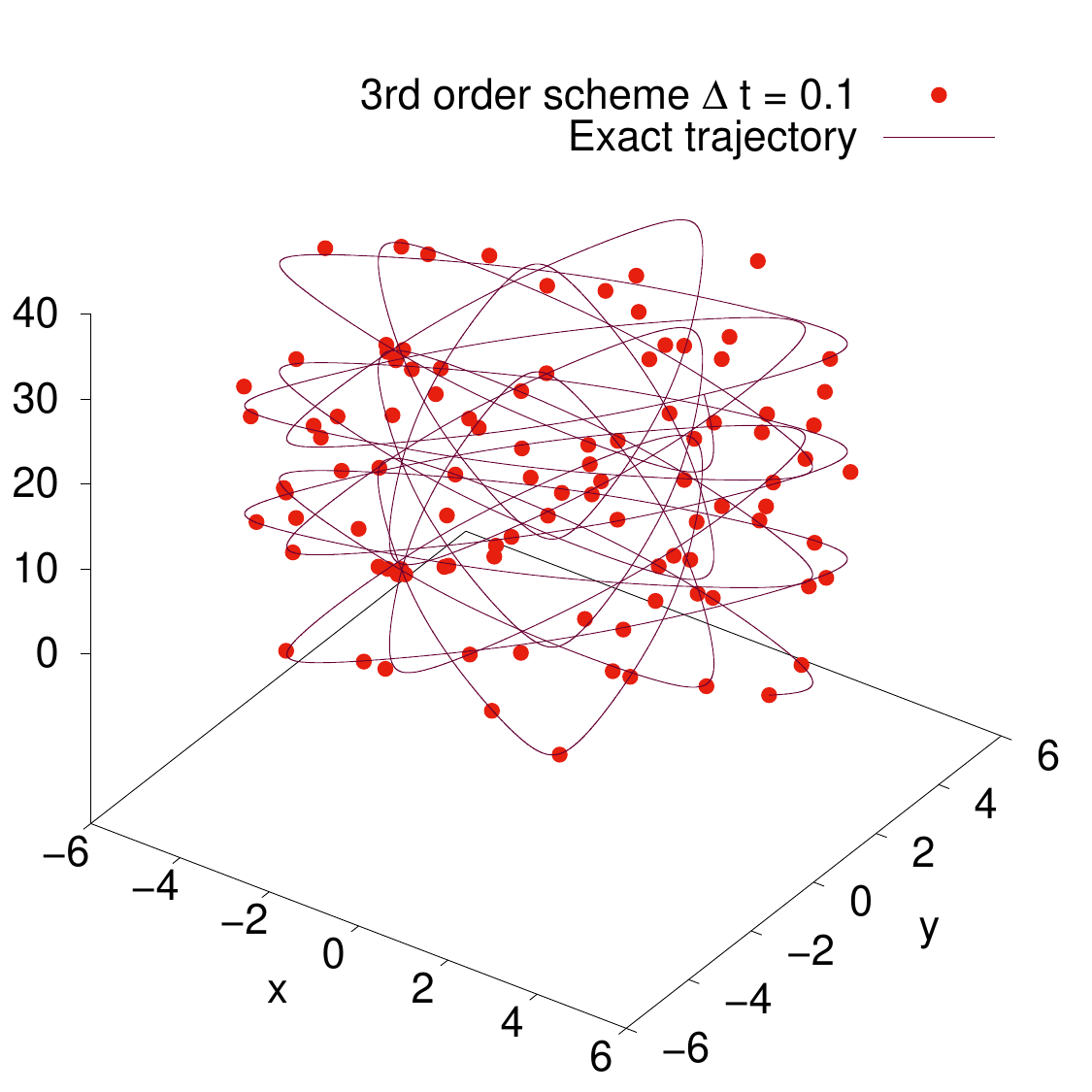} &    
\includegraphics[width=7.cm]{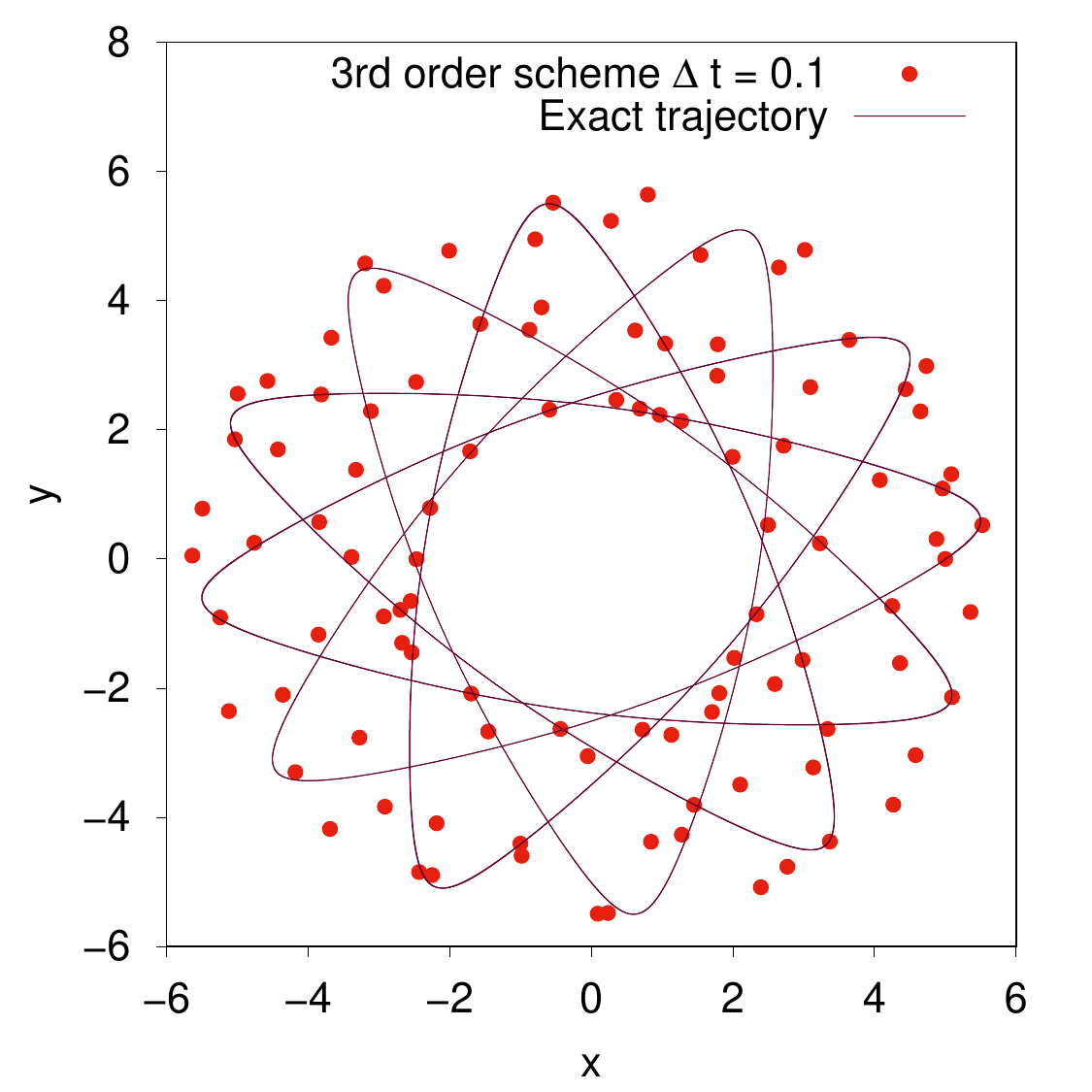} 
\\
(a) $3D$   & (b)  $2D$ 
\end{tabular}
\caption{\label{fig0:1}
{\bf One single particle motion without electric field.} Space
trajectory (a) in three dimension,  (b) $x-y$ two dimensional
projection obtained with  fixed time steps $\Delta t=0.1$ with
third-order scheme \eqref{scheme:4-1}-\eqref{scheme:4-7}, plotted as
functions from top to bottom as $\eps=10^{-1}$ and $10^{-2}$.}
 \end{center}
\end{figure}

\begin{figure}
\begin{center}
 \begin{tabular}{cc}
\includegraphics[width=7.cm]{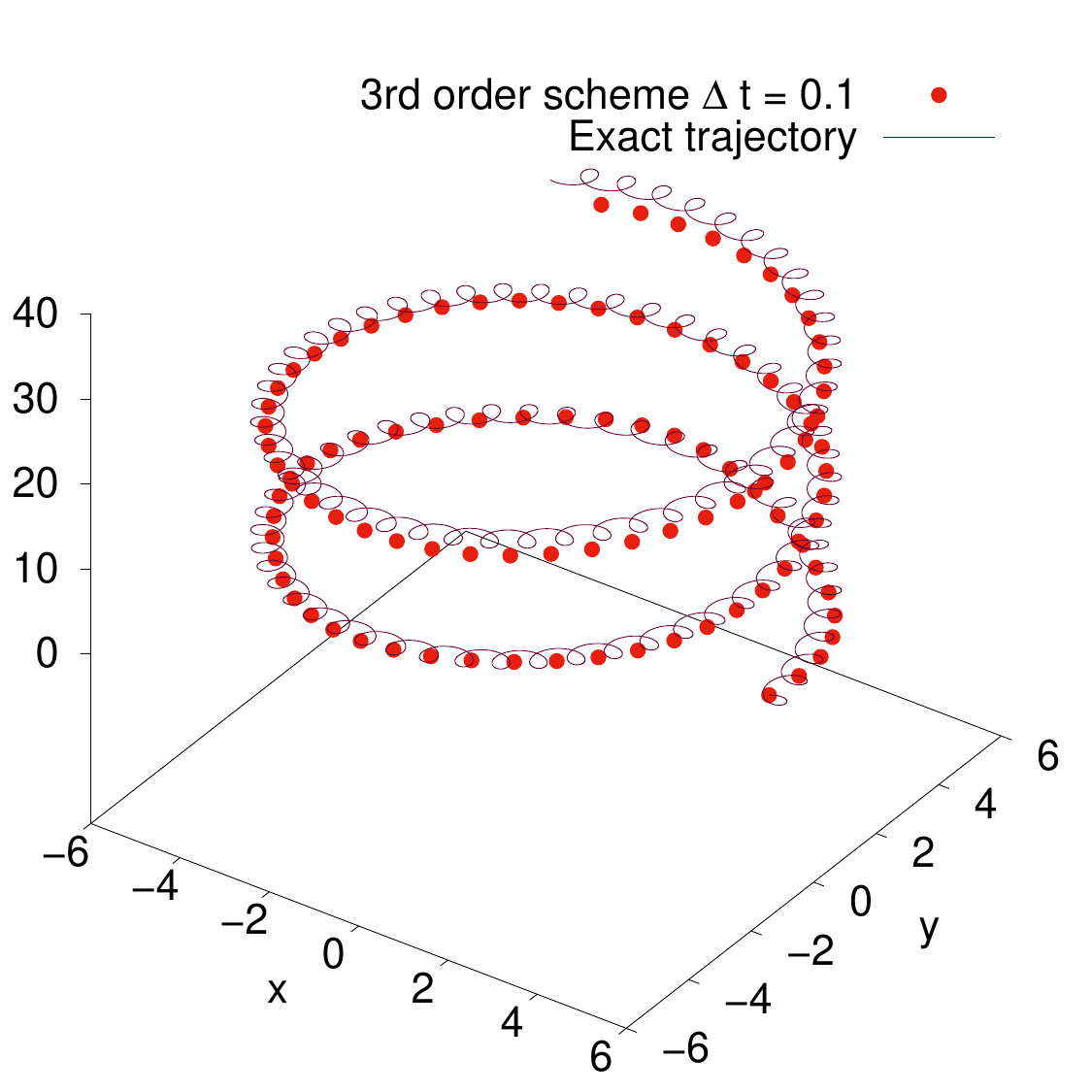} &    
\includegraphics[width=7.cm]{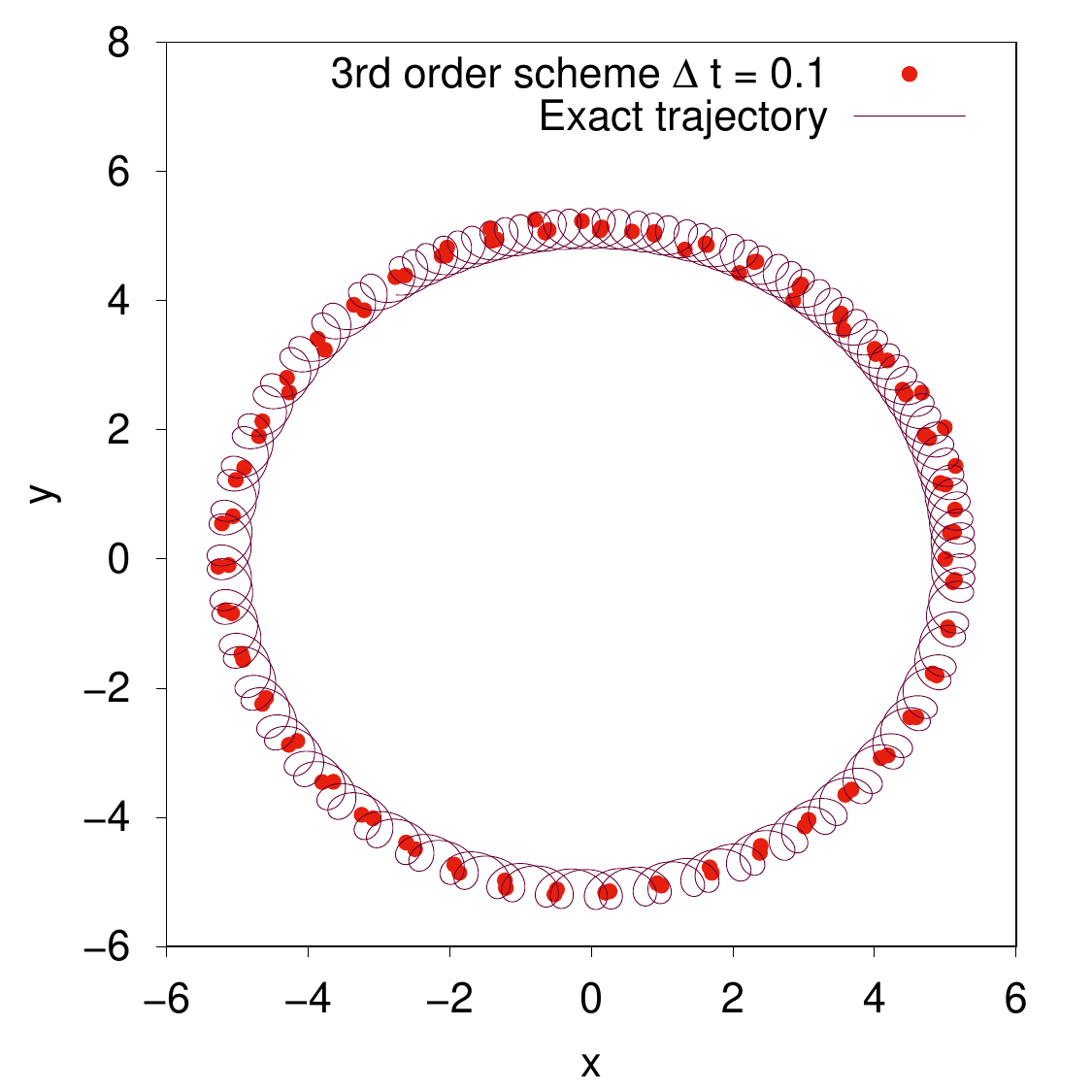} 
\\
\includegraphics[width=7.cm]{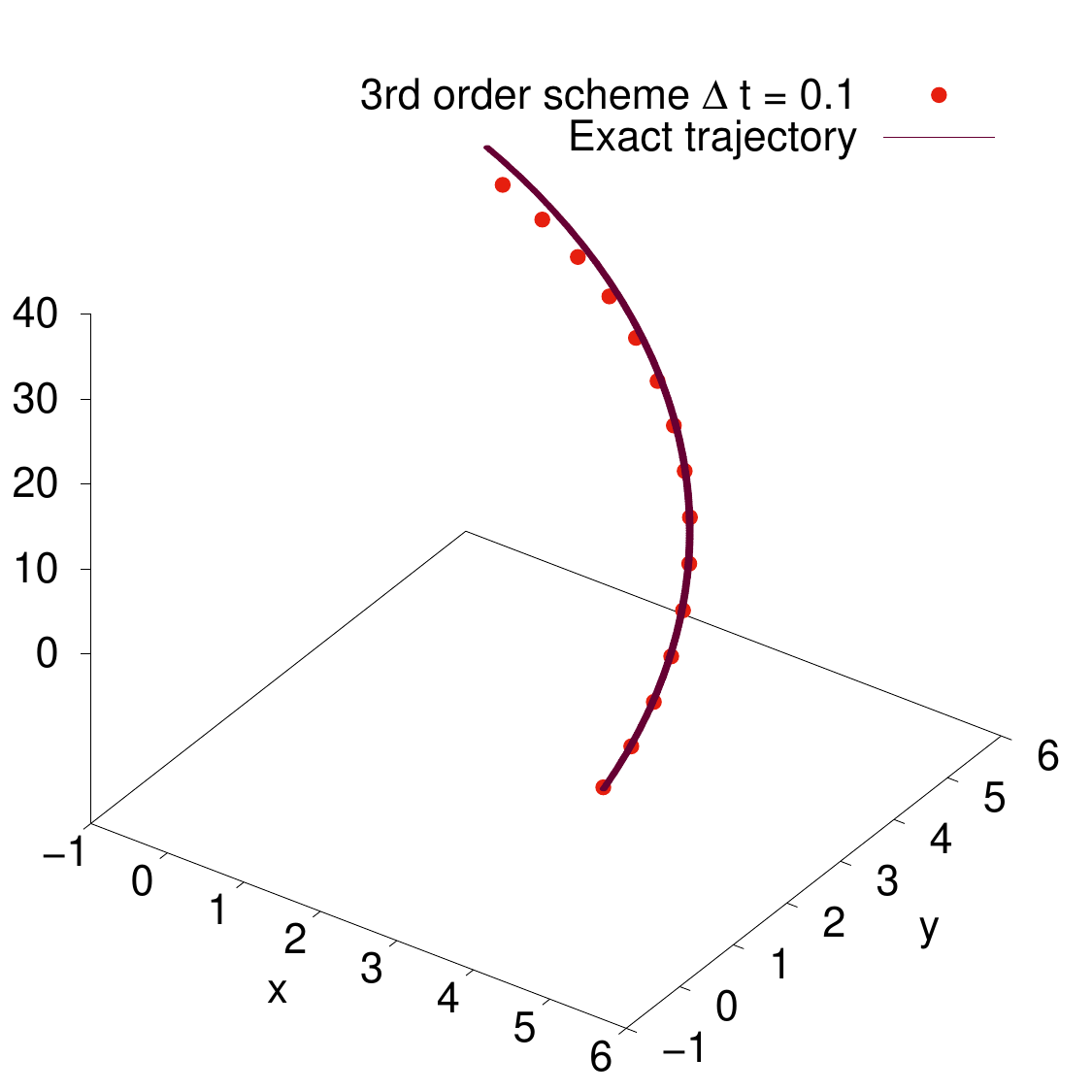} &    
\includegraphics[width=7.cm]{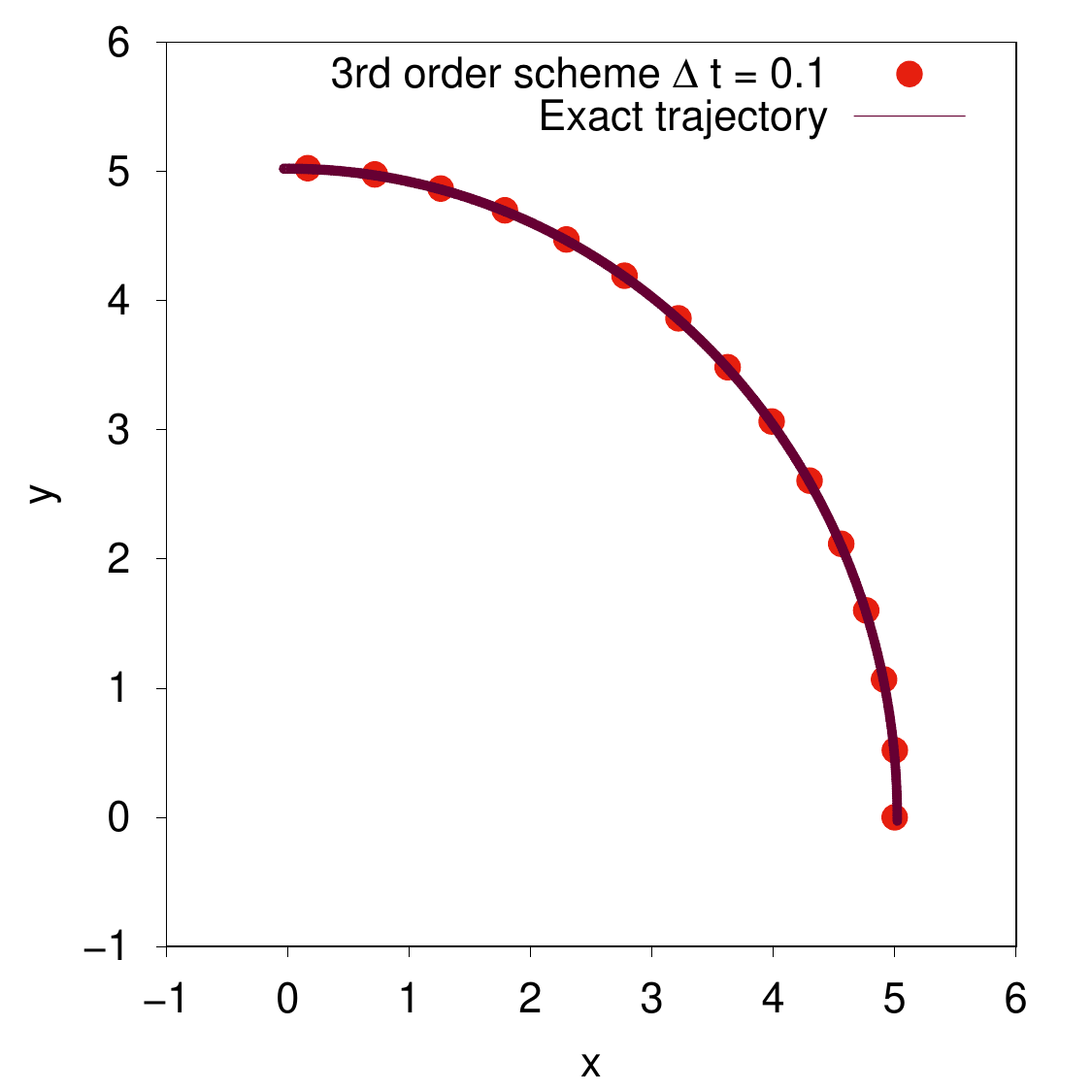} 
\\
(a) $3D$   & (b)  $2D$ 
\end{tabular}
\caption{\label{fig0:2}
{\bf One single particle motion without electric field.} Space
trajectory (a) in three dimension,  (b) $x-y$ two dimensional
projection obtained with  fixed time steps $\Delta t=0.1$ with
third-order scheme \eqref{scheme:4-1}-\eqref{scheme:4-7}, plotted as
functions from top to bottom as $\eps=10^{-3}$ and $10^{-4}$.}
 \end{center}
\end{figure}

\subsection{The Vlasov-Poisson system }
We now consider the Vlasov-Poisson system \eqref{eq:vlasov}, then
ignoring the contribution of boundary conditions or assuming that the
density is concentred far from the boundary,  the total energy
$\mE(t)$ is given by
$$
\mathcal{E}(t) \,:=\,
\iint_{\Omega\times\RR^3} f^\eps(t,\bx,\bv)\,\frac{\|\bv\|^2}{2}\,\dD\bx\,\dD \bv
\,+\,\frac{1}{2}\int_{\Omega} \|\bE^\eps(t,\bx) \|^2 \dD\bx
$$
and is conserved with time. Observe for the asymptotic model
\eqref{eq:drift},  the same energy can be defined as

$$
\mathcal{E}(t) \,:=\,
\iint_{\Omega\times\RR^3} F^\eps(t,\bx,e_\perp,v_\mypar)\,\left( e_\perp +
  \frac{|v_\mypar|^2}{2}\right)\,\dD\bx\dD e_\perp \dD v_\mypar
\,+\,\frac{1}{2}\int_{\Omega} \|\bE^\eps(t,\bx) \|^2 \dD\bx.
$$
As far as smooth solutions are concerned, the total energy is
preserved by both the original $\eps$-dependent model and by the
asymptotic model \eqref{eq:drift}. One goal of our experimental
observations is to check that despite the fact that our scheme
dissipates some parts of the velocity variable to reach the asymptotic
regime corresponding to \eqref{eq:drift} it does respect this
conservation.

Furthermore, assuming that $b$ does not depend on time, we define  the adiabatic variable given by
$$
\mu(t)=\int_{\Omega}\int_{\RR^3}f^\eps(t,\bx,\bv)\,\frac{\|\bv_\perp\|^2}{2b(\bx_\perp)}
\,\dD\bx\,\dD\bv\,.
$$
In contrast to the energy, an essentially exact conservation of the
adiabatic variable is a sign that we have reached the limiting
asymptotic regime since it does not hold for the original model but
does for the asymptotic~\eqref{eq:drift} as $b$ is time-independent and
$\bE$ is curl-free. Observe that, since $b$ is not homogeneous, even in the
asymptotic regime the kinetic and potential parts of the total
energy are not preserved separately, but the total energy
corresponding to the Vlasov-Poisson system is still
preserved.

\subsubsection*{Diocotron instability in a cylinder.}
Here we choose $\Omega=D\times (0,L_z)$ with $D=D(0,6)$ the disk centered at the origin with radius $6$ and
$L_z=1$. Our simulations start with an initial data that is Maxwellian in velocity and whose macroscopic density is the sum of two Gaussians, explicitly 
$$
f_0(\bx,\bv) \,=\, \frac{n_0(\bx)}{(2\pi)^{3/2}}\,\exp\left(-\frac{\|\bv\|^2}{2}\right),
$$ 
where $n_0$ is chosen as
$$
n_0(\bx) = \left\{
\begin{array}{ll}
\ds n_0\,\left(1\,+\, \alpha (\cos(\theta)+5\,\cos(2\pi\,k_z\,z))\right)  & {\rm \,if\,} 6 \leq
                                           r_\perp\leq 7,
\\ \,\\
\ds 0 &{\rm \,else,\,} 
\end{array} \right.
$$
with $r_\perp=\| \bx_\perp\|$, $\theta=\arctan(y/x)$,  $n_0=4000$, $k_z=3$
and $\alpha=0.001$. Moreover, in the Poisson equation
\eqref{eq:poisson} we take $\rho_0=0$. The parameter $\eps$ is chosen as $\eps=0.05$, where the asymptotic regime is
relevant. We compute numerical solutions to the
Vlasov-Poisson system \eqref{eq:vlasov}  with the third-order scheme
\eqref{scheme:4-1}-\eqref{scheme:4-7} and time step $\Delta t=0.1$. We first run one set of numerical simulations homogeneous magnetic field
$b=1$.

In Figure~\ref{fig2:0} we present the time evolution of the relative
variation of energy and adiabatic variable.  For instance,
$$
\Delta \mE_\alpha = \frac{\mE_\alpha(t) -
  \mE_\alpha(0)}{\mE_\alpha(0)}, \quad \alpha\in\{k, p,t\},
$$
where $k$ denotes the kinetic energy, $p$ is the potential energy and
$t$ is the total energy. Notice that total energy  is conserved at the continuous level but not with our numerical
scheme. However, we show that all these features are
captured satisfactorily by our scheme even on long time evolutions
with a large time step and despite its dissipative implicit nature in the asymptotic regime when $\eps=0.05$.
  
\begin{figure}
\begin{center}
 \begin{tabular}{cc}
\includegraphics[width=7.cm,height=7.cm]{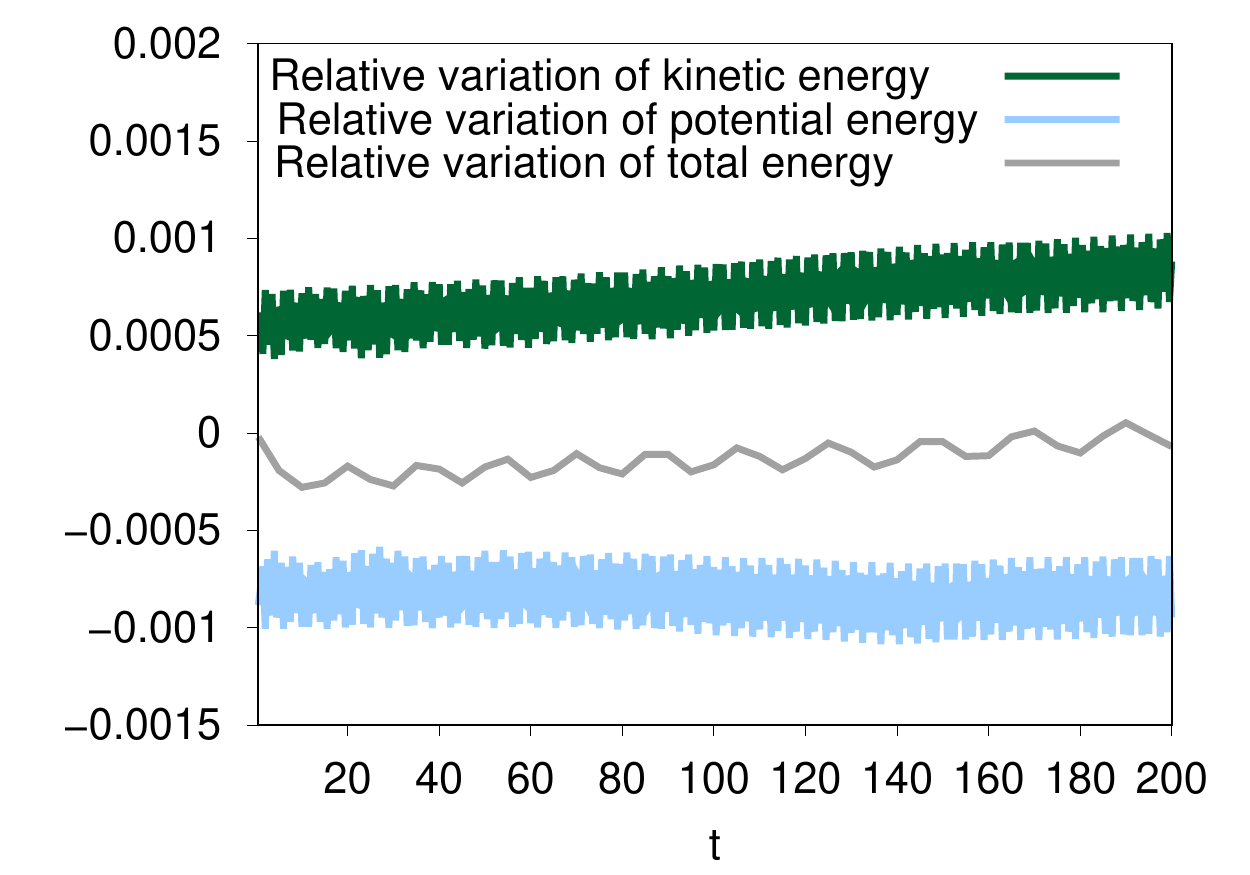} &    
\includegraphics[width=7.cm,height=7.cm]{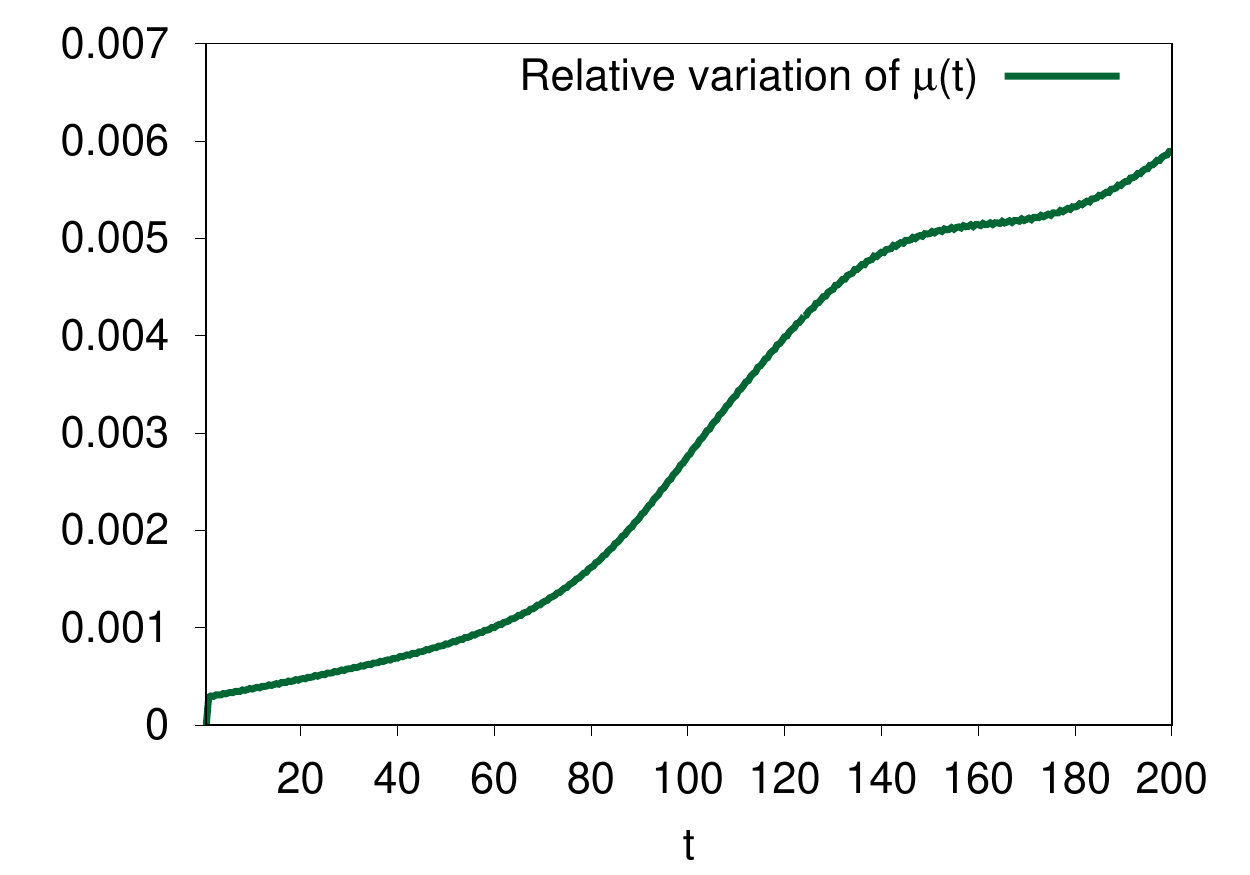} 
\\
(a) $\Delta\mE(t)$ &(b) $\Delta\mu(t)$
\end{tabular}
\caption{\label{fig2:0}
{\bf Diocotron instability in a cylinder.}  Time evolution of total energy and
adiabatic invariant  with  $\eps=0.05$ obtained using
\eqref{scheme:4-1}-\eqref{scheme:4-7} with $\Delta t=0.5$.}
 \end{center}
\end{figure}

In Figure~\ref{fig2:1}  we visualize the corresponding
dynamics by presenting some snapshots of the time evolution of the
macroscopic charge density. We take $\eps=0.05$ such that the magnetic
field is sufficiently large to provide a good confinement of the
macroscopic density. Here, we expect similar results as for the two
dimensional 
diocotron instability where seven vortices are generated \cite{bibFilbet_Rodrigues, bibFilbet_Rodrigues2}. 

\begin{figure}
\begin{center}
 \begin{tabular}{cc}
\includegraphics[width=7.cm]{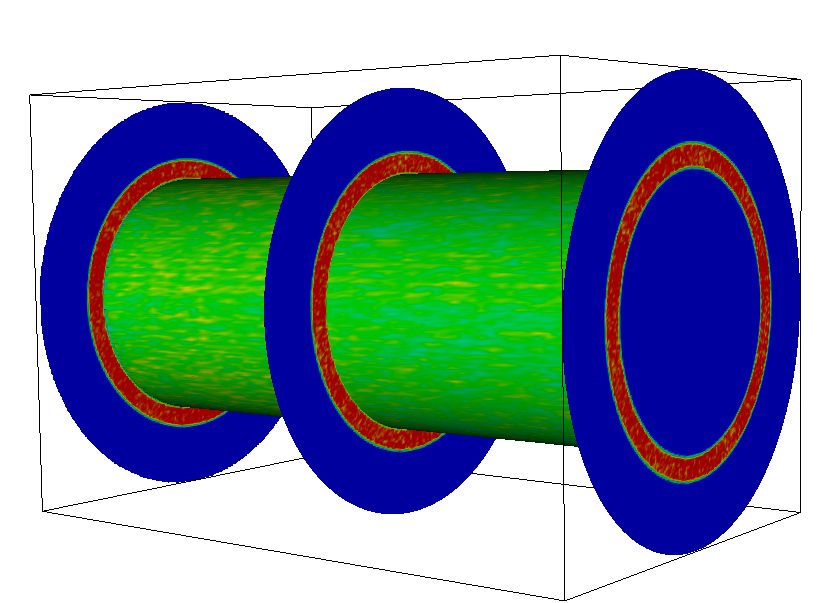}&    
\includegraphics[width=7.cm]{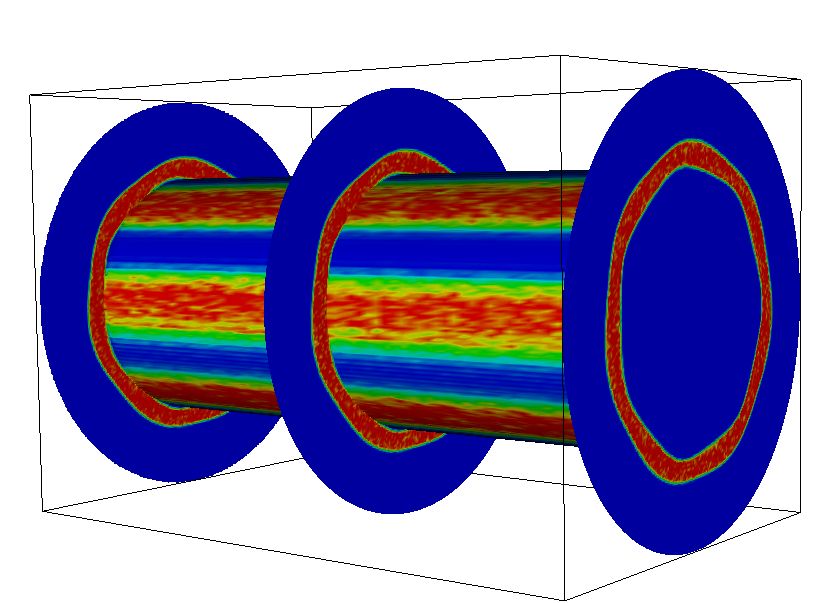} 
\\
$t=000$  & $t=040$
\\
\includegraphics[width=7.cm]{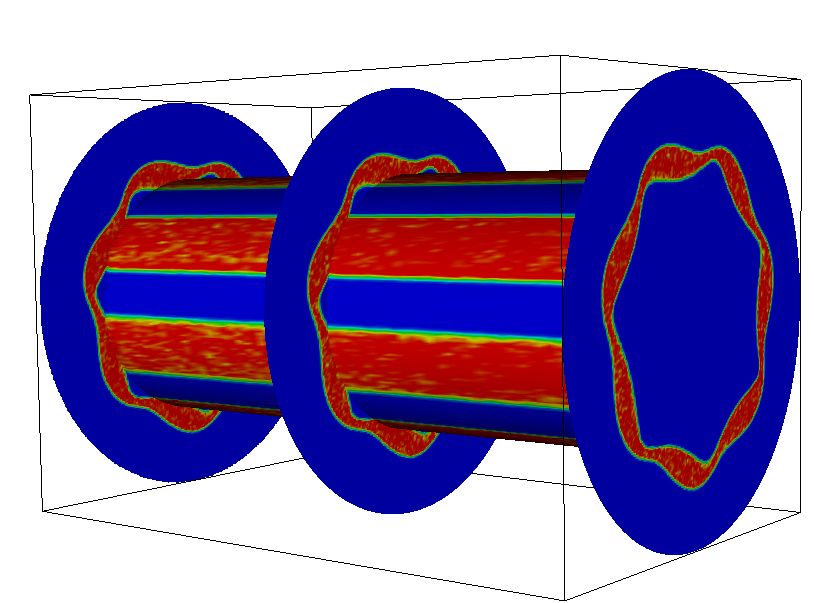}&
\includegraphics[width=7.cm]{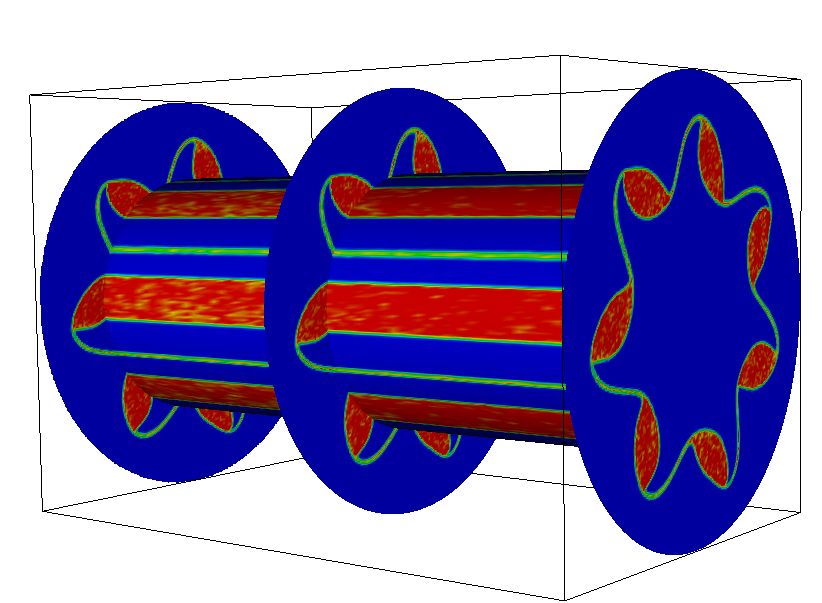} 
\\
$t=080$ &$t=120$  
\\
\includegraphics[width=7.cm]{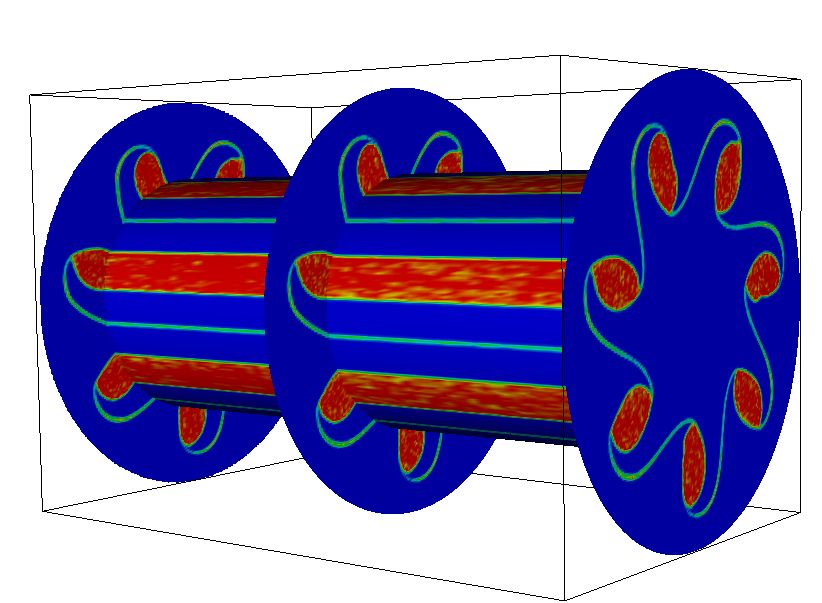}&
\includegraphics[width=7.cm]{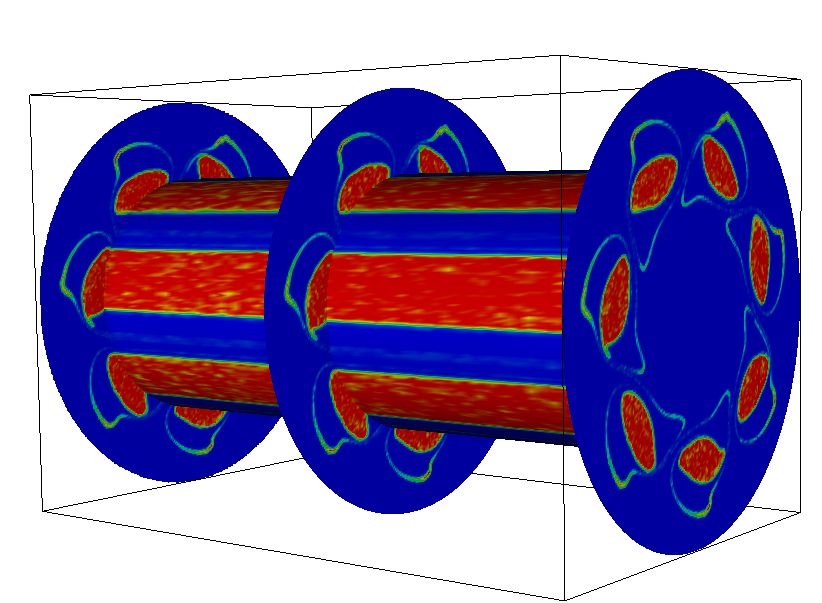} 
\\
$t=140$ & $t=200$
\end{tabular}
\caption{\label{fig2:1}
{\bf Diocotron instability in a cylinder.}  Snapshots of the time evolution of the macroscopic charge density $\rho$ when $\eps=0.05$, obtained using
\eqref{scheme:4-1}-\eqref{scheme:4-7} with $\Delta t=0.5$ .}
 \end{center}
\end{figure}

\subsubsection*{Fusion of vertices in a $D$-shaped domain}
We consider now  a D-shaped domain in the plane orthogonal to the
magnetic field $D$ presented in Section IV of~\cite{bibmiller} and
depicted in Figure~\ref{fig:Dshape} (a). The mapping from curvilinear coordinates  $\mathbf{\xi} = (\xi_1,\xi_2)$ to physical coordinates $\bx_\perp = (x, y)$ is given by

\begin{equation*}
\left\{
  \begin{array}{l}
    x \, = \, x_c + \xi_1\, \cos\left(\xi_2 \,+\,
    \arcsin(0.416)\sin(\xi_2)\right),
\\[3mm]
    y\, = \,  y_c \,+\, 1.66 \,\xi_1\,\sin(2\pi\xi_2),
  \end{array}\right.
\end{equation*} 
for $0\leq \xi_1 \leq  R_0$ and $0 \leq \xi_2 \leq  2\pi$ with
$(x_c,y_c)=(0,0)$ and $R_0=10$.

\begin{figure}
\begin{center}
 \begin{tabular}{cc}
\includegraphics[width=7.cm,height=6.cm]{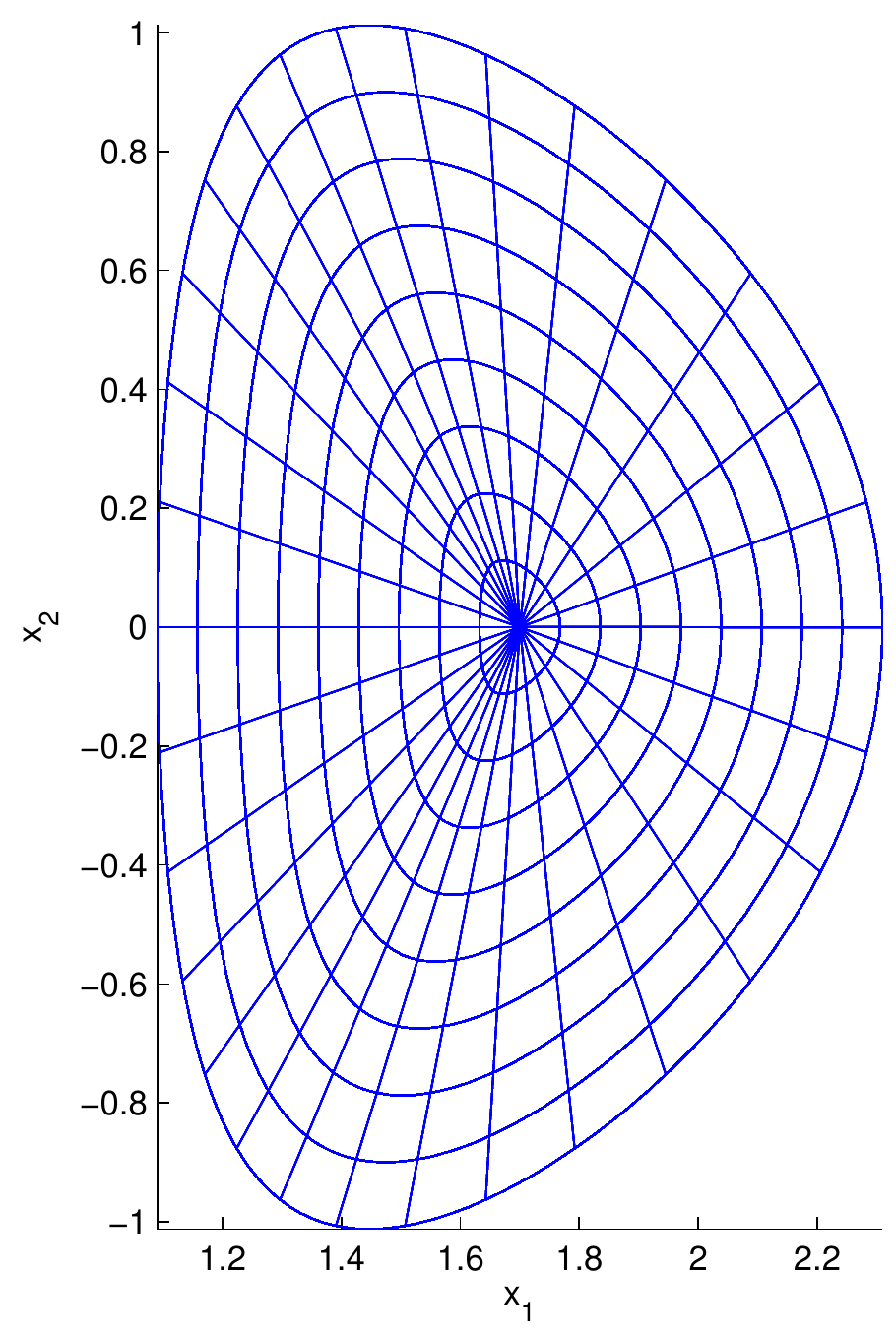} &   
\includegraphics[width=7.cm,height=6.cm]{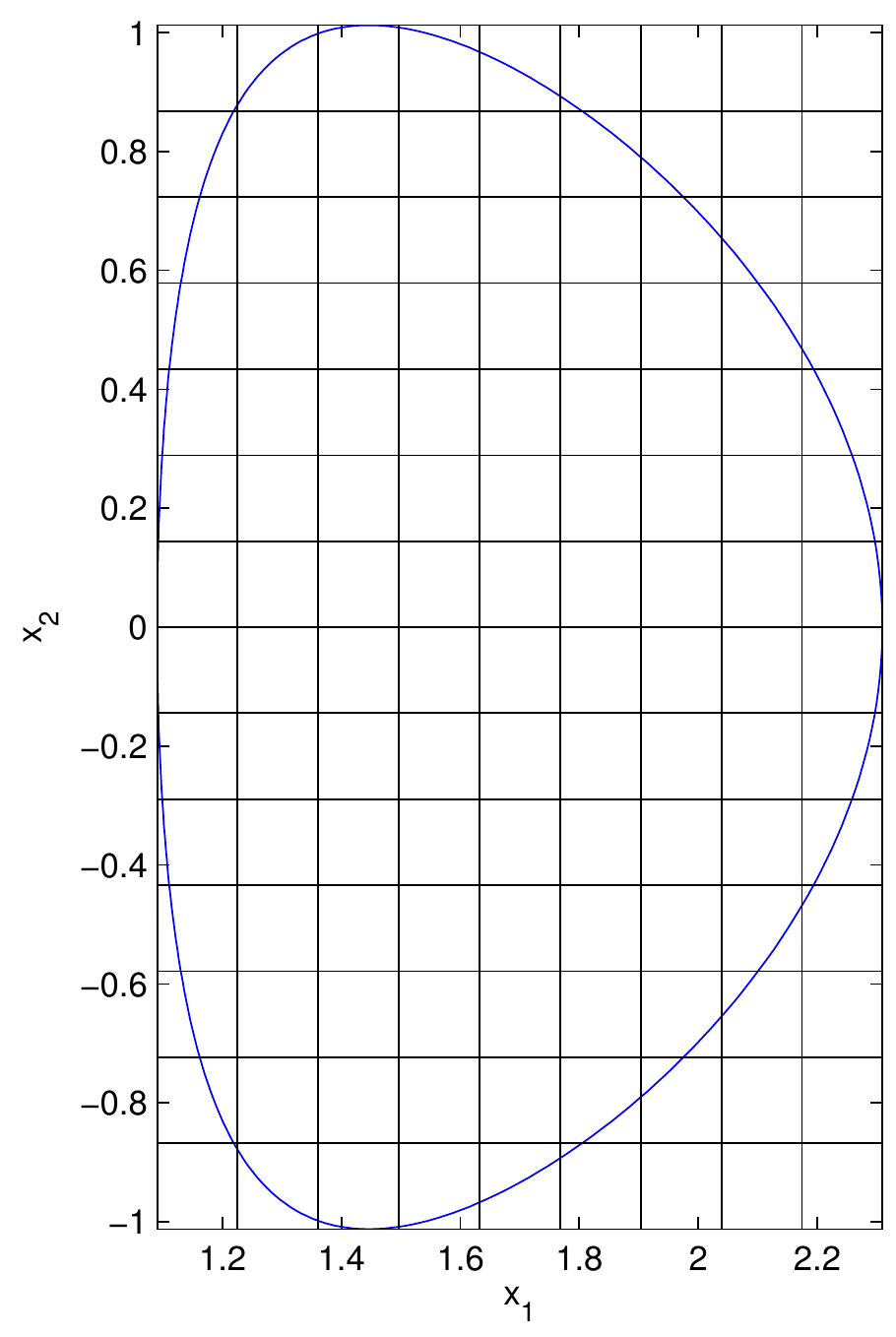} 
\\
  (a)  &   (b)
 \end{tabular}
\caption{\label{fig:Dshape}D-shaped domain~\cite{bibmiller}. (a) Constant lines in coordinates $\xi = (\xi_1,\xi_2)$; (b) D-shaped domain embedded in Cartesian mesh.}
 \end{center}
\end{figure}
Our simulations start with an initial data that is Maxwellian in
velocity and whose macroscopic density is the sum of two Gaussians in
the perpendicular plane to the magnetic field and a perturbed constant
homogeneous density in the parallel direction to the magnetic field, explicitly 
$$
f_0(\bx,\bv) \,=\, \frac{n_0(z)}{8\pi^2 r_0^2} \left[
  \exp\left(-\frac{\|\bx_\perp-\bx_{0\perp}\|^2}{2r_0^2}\right) + \exp\left(-\frac{\|\bx_\perp+\bx_{0\perp}\|^2}{2\,r_0^2}\right)\right]\, \exp\left(-\frac{\|\bv\|^2}{2}\right),
$$ 
with $\bx_{0\perp}=(3/2,-3/2,0)$, $r_0=3$ and the density
$n_0(z)=5000\,(1+\alpha\cos(k_z\,z))$ with $k_z=2\pi/L_z$. 

We choose  a time-independent inhomogeneous magnetic field
$$
b\,:\quad \RR^2\to\RR\,,\qquad \bx\,\mapsto\,\frac{20}{\sqrt{20^2-\|\bx_\perp\|^2}},
$$
that is, radial increasing with value one at the origin.

In Figure~\ref{fig3:0} we present again the time evolution of the relative
variation of energy and adiabatic variable. In this regime, the limit
model \eqref{eq:drift} makes sense and it is expected that both the
total energy $\mE$ and the adiabatic invariant $\mu$ are
conserved. Once again the numerical results are satisfactory since
even for large times, the relative variations are of order $10^{-3}$.
\begin{figure}
\begin{center}
 \begin{tabular}{cc}
\includegraphics[width=7.cm,height=7.cm]{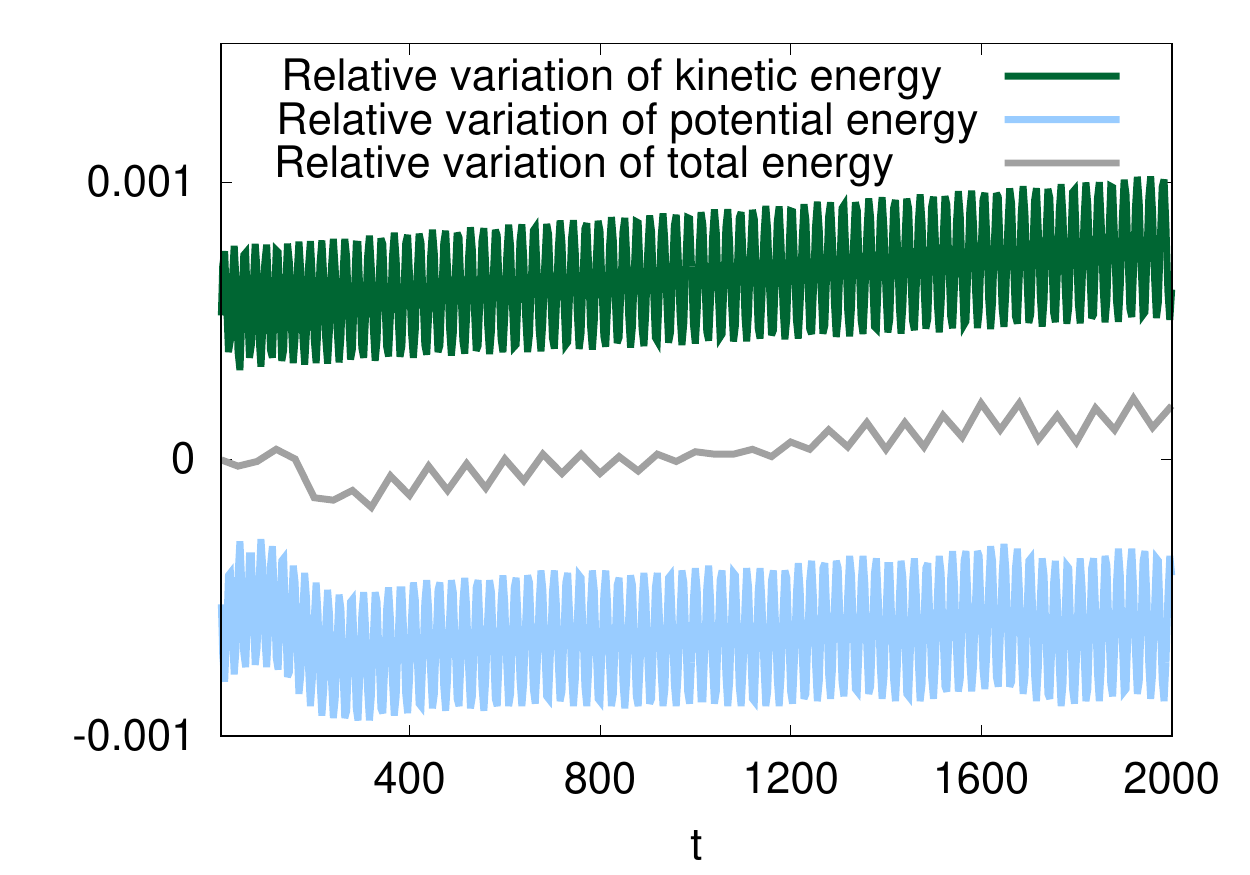} &    
\includegraphics[width=7.cm,height=7.cm]{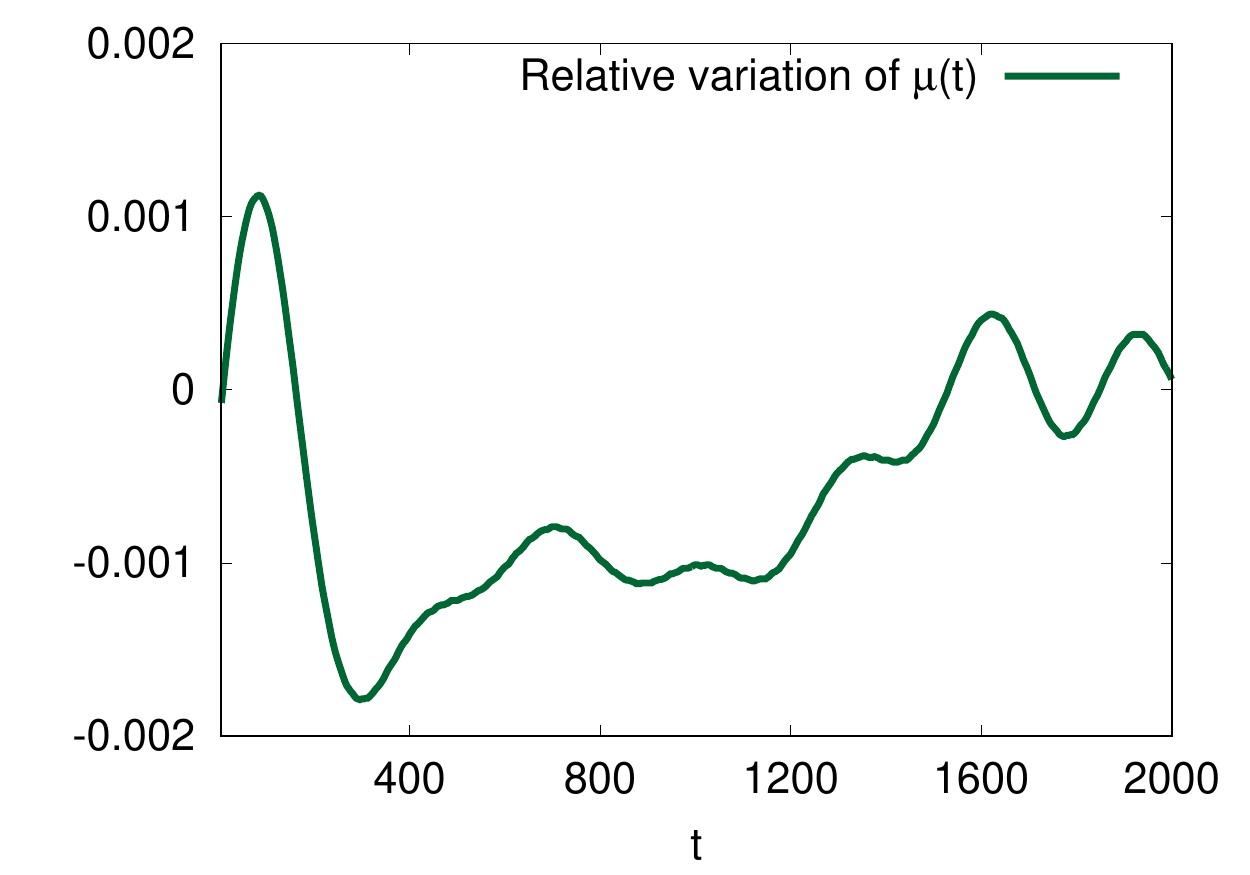} 
\\
(a) $\Delta\mE_\alpha(t) =
   \frac{\mE_\alpha(t)-\mE_\alpha(0)}{\mE_\alpha}$ &(b) $\Delta\mu(t)=
                                                     \frac{\mu(t)-\mu(0)}{\mu(0)}$
\end{tabular}
\caption{\label{fig3:0}
{\bf Fusion of vertices in a $D$-shaped domain.}  Time evolution of (a)
total energy, kinetic energy and potential energy (b)
adiabatic invariant  with  $\eps=0.01$ obtained using
\eqref{scheme:4-1}-\eqref{scheme:4-7} with $\Delta t=0.5$.}
 \end{center}
\end{figure}

In Figure~\ref{fig3:1}  we visualize the corresponding
dynamics by presenting some snapshots of the time evolution of the
macroscopic charge density. We take $\eps=0.01$ such that the magnetic
field is sufficiently large to provide a good confinement of the
macroscopic density.  At time $t=30$, some filament structures can be
identified. These filaments are observed more clearly for larger
times. Since the intensity of the magnetic field is sufficiently
large, the plasma is well confined and the two vertices merge, whereas
some the filaments persist and generate a ``halo'' which propagates in
the domain. 

\begin{figure}
\begin{center}
 \begin{tabular}{cc}
\includegraphics[width=7.0cm]{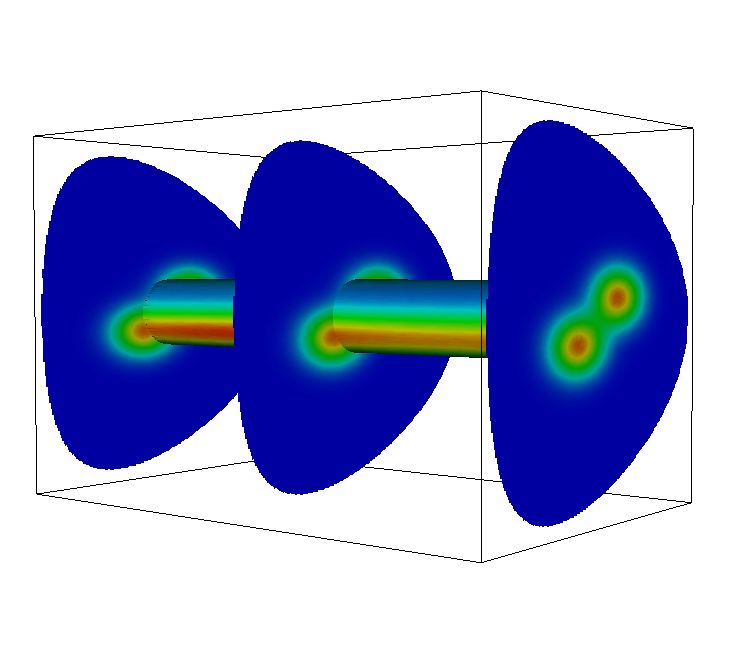}&    
\includegraphics[width=7.0cm]{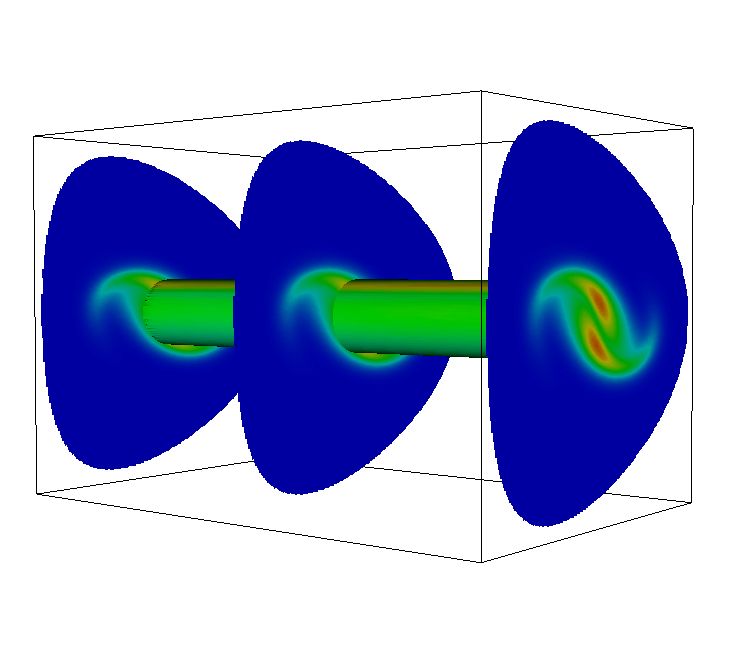}    
\\
$t=0000$  & $t=0160$
\\
\includegraphics[width=7.0cm]{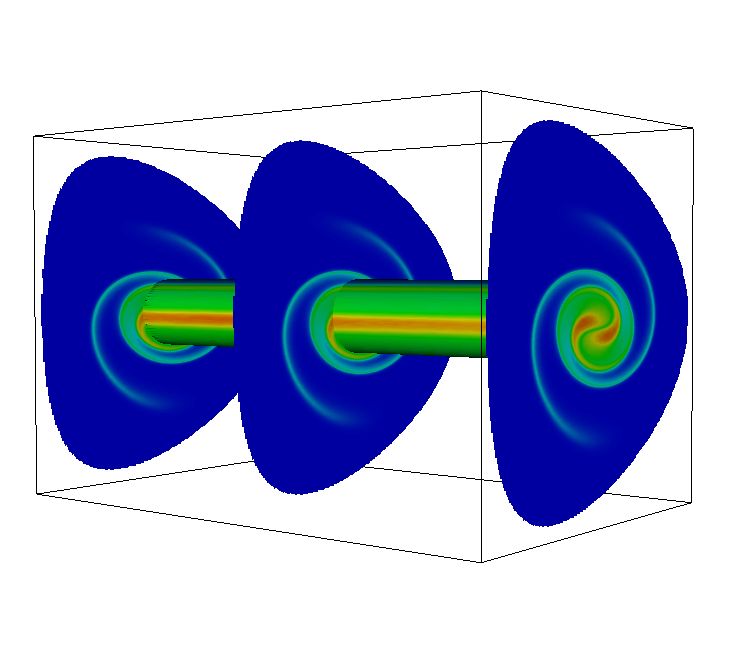}& 
\includegraphics[width=7.0cm]{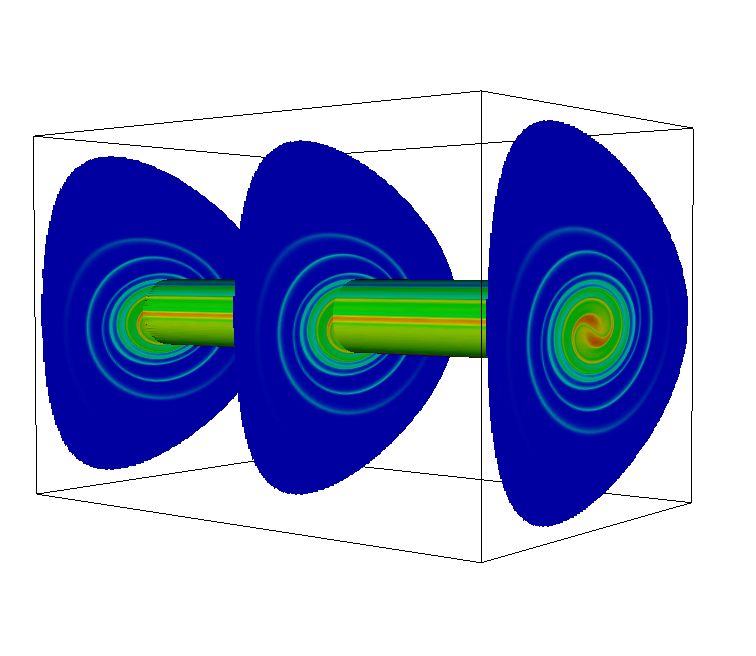}
\\
$t=0480$&$t=0640$ 
\\
\includegraphics[width=7.0cm]{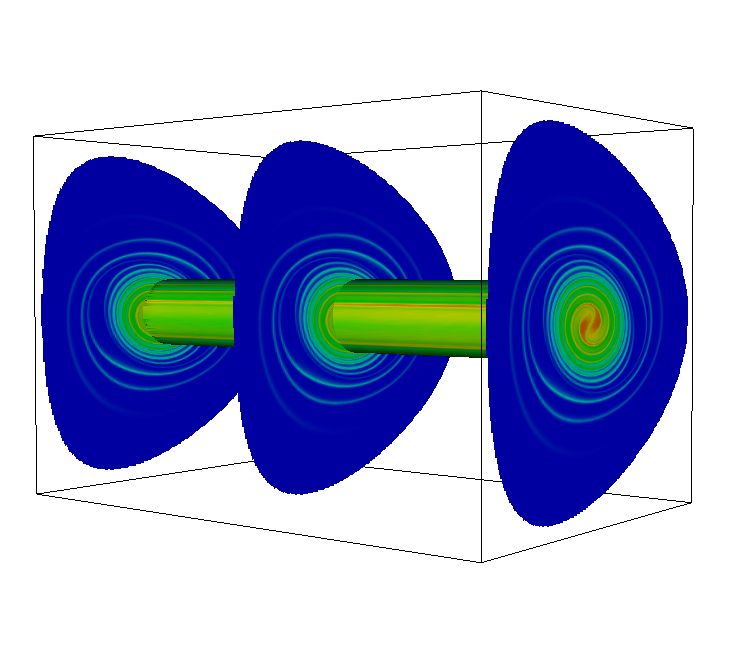}&
\includegraphics[width=7.0cm]{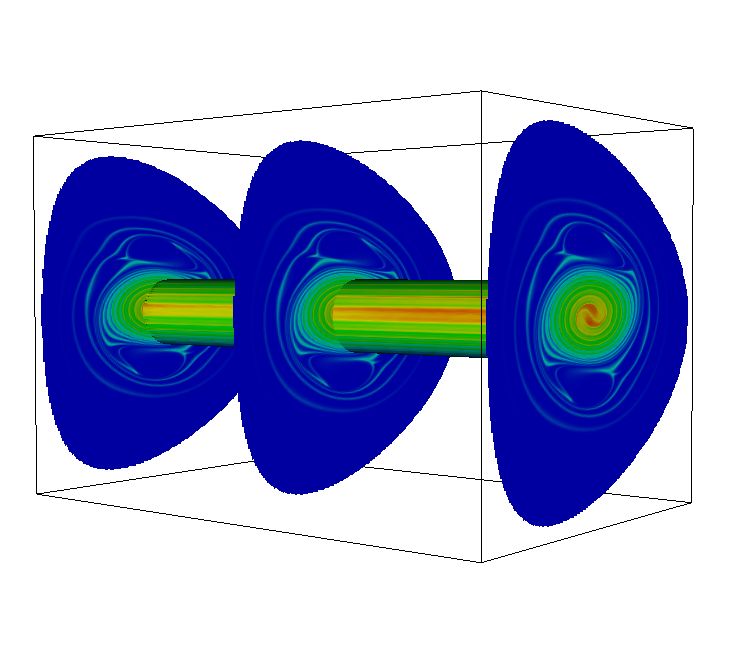} 
\\
$t=0960$  &$t=1920$  
\end{tabular}
\caption{\label{fig3:1}
{\bf Fusion of vertices in a $D$-shaped domain.}  Snapshots of the time evolution of the macroscopic charge density $\rho$ when $\eps=0.01$, obtained using
\eqref{scheme:4-1}-\eqref{scheme:4-7} with $\Delta t=0.5$ .}
 \end{center}
\end{figure}

\section{Conclusion and perspective}
\label{sec:6}
\setcounter{equation}{0}

In the present paper we have proposed a class of semi-implicit time discretization
techniques for particle-in cell simulations of the three dimensional
Vlasov-Poisson system. The main feature of our approach is to
guarantee the accuracy and stability on slow scale variables even when
the amplitude of the magnetic field becomes large including cases with
non homogeneous magnetic fields and coarse time grids. Even on large
time simulations the obtained numerical schemes also provide an
acceptable accuracy on physical invariants (total energy for any
$\eps$, adiabatic invariant when $\eps\ll 1$) whereas fast scales are automatically filtered when the time step is large compared to $\eps$.

As a theoretical validation we have proved that the discrete
trajectories remain bounded for the semi-implicit schemes and for
$\eps\ll 1$, the schemes is consistent with the asymptotic model and
preserve the order of accuracy with respect to $\Delta t$. From a practical point of view, the next natural step would be to consider the genuinely three-dimensional Vlasov-Poisson system taking into
account curvature effects.

%%%%%%%%%%%%%%%%%%%%%%%%%%%%%%%%%%%%%%%%%%%%%%%%%%%%%%%%%%%%%%%%%%%%%%%%%%%%%%%%%%%%%%%%%%%%
%%%%%%%%%%%%%%%%%%%%%%%%%%%%%%%%%%%%%%%%%%%%%%%%%%%%%%%%%%%%%%%%%%%%%%%%%%%%%%%%%%%%%%%%%%%%
\section{Acknowledgements}
\label{sec:7}
\setcounter{equation}{0}

Francis Filbet was supported by the EUROfusion Consortium and has received funding from the Euratom research and training programme 2014-2018 under grant
agreement No. 633053. The views and opinions expressed herein do not
necessarily reflect those of the European Commission.

{Chang Yang was supported by National Natural Science Foundation of China (Grant No. 11401138).}
%%%%%%%%%%%%%%%%%%%%%%%%%%%%%%%%%%%%%%%%%%%%%%%%%%%%%%%%%%%%%%%%%%%%%%%%%%%%%%%%%%%%%%%%%%%%
%%%%%%%%%%%%%%%%%%%%%%%%%%%%%%%%%%%%%%%%%%%%%%%%%%%%%%%%%%%%%%%%%%%%%%%%%%%%%%%%%%%%%%%%%%%%

\end{document}